\documentclass[preprint,10pt]{elsarticle}

\usepackage{amssymb}
\usepackage{amsthm}
\usepackage{amsmath}
\usepackage{graphicx}
\usepackage{chngcntr}
\usepackage{apptools}

\numberwithin{equation}{section}
\numberwithin{figure}{section}
\theoremstyle{plain}
\newtheorem{thm}{\protect\theoremname}[section]
\newenvironment{lyxlist}[1]
{\begin{list}{}
{\settowidth{\labelwidth}{#1}
 \setlength{\leftmargin}{\labelwidth}
 \addtolength{\leftmargin}{\labelsep}
 }}
{\end{list}}
  \theoremstyle{plain}
  \newtheorem{prop}[thm]{\protect\propositionname}
  \theoremstyle{remark}
  \newtheorem{rem}[thm]{\protect\remarkname}
  \theoremstyle{plain}
  \newtheorem{cor}[thm]{\protect\corollaryname}
	
  \providecommand{\corollaryname}{Corollary}
  \providecommand{\propositionname}{Proposition}
  \providecommand{\remarkname}{Remark}
\providecommand{\theoremname}{Theorem}

\journal{Journal of Differential Equations}

\begin{document}

\begin{frontmatter}



\title{Saddle-Node Bifurcation of Periodic Orbits for a Delay Differential Equation}

\author[SzG]{Szandra Guzsv\'any}
\ead{gszandra@math.u-szeged.hu}
\address[SzG]{Bolyai Institute, University of Szeged, 1 Aradi v. tere, Szeged, Hungary}

\author[GV]{Gabriella Vas\corref{cor1}}
\ead{vasg@math.u-szeged.hu}
\address[GV]{MTA-SZTE Analysis and Stochastics Research Group, Bolyai Institute, University of Szeged, 1 Aradi v. tere, Szeged, Hungary}
\cortext[cor1]{Corresponding author}

\begin{abstract}
We consider the scalar delay differential equation 
\[
\dot{x}(t)=-x(t)+f_{K}(x(t-1))
\]
with a nondecreasing feedback function $f_{K}$ depending on a parameter
$K$, and we verify that a saddle-node bifurcation of periodic orbits
takes place as $K$ varies. 

The nonlinearity $f_{K}$ is chosen so that it has two unstable fixed
points (hence the dynamical system has two unstable equilibria), and
these fixed points remain bounded away from each other as $K$ changes.
The generated periodic orbits are of large amplitude in the sense
that they oscillate about both unstable fixed points of $f_{K}$.
\end{abstract}


\end{frontmatter}


\section{Introduction}

Numerous scientific works have studied the existence and the bifurcation
of periodic orbits for delay differential equations, see the books
\cite{Diekmann,Erneux,HVL} and the survey paper \cite{Walther2}
of Walther. In paper \cite{Krisztin-1}, Krisztin gives a detailed
summary on the known results for equations of the special form 
\begin{equation}
\dot{x}(t)=-x(t)+f_{K}(x(t-1)),\label{Eq}
\end{equation}
where $f_{K}$ is monotone nonlinearity. Hopf-bifurcation is a widely
studied phenomenon \cite{Walther2}. A well-known example is due to
Krisztin, Walther and Wu: periodic orbits of \eqref{Eq} arise via
a series of Hopf-bifurcations for strictly monotone increasing nonlinearities,
e.g., for $f_{K}\left(x\right)=K\tanh(x)$ or for $f_{K}\left(x\right)=K\tan^{-1}(x)$
as $K$ increases, see \cite{Krisztin-Walther,Krisztin-Walther-Wu,Krisztin-Wu}.
Other types of bifurcations involving periodic orbits are rarely studied.
An interesting example is given by Walther in \cite{Walther}: he
studies a delay equation coming from a prize model, and he shows the
bifurcation of periodic orbits with small amplitudes and with periods
descending from infinity. To the best of the authors' knowledge, no
one has verified saddle-node bifurcation of periodic orbits neither
for \eqref{Eq} nor for other types of delay differential equations. 

In this paper we consider \eqref{Eq} in the so-called positive feedback
case; $f_{K}$ is supposed to be a nondecreasing continuous function
such that 
\[
f_{K}|_{(-\infty,-1-\varepsilon]}=-K,\quad f_{K}|_{[-1,1]}=0\quad\mbox{and}\quad f_{K}|_{[1+\varepsilon,\infty)}=K,
\]
where $\varepsilon$ is a fixed positive number and $K$ is the bifurcation
parameter. For technical simplicity, we define $f_{K}$ to be a piecewise
linear continuous function: 
\[
f_{K}(x)=\frac{K}{\varepsilon}\left(x+1\right)\quad\mbox{for }x\in(-1-\varepsilon,-1)
\]
and 
\[
f_{K}(x)=\frac{K}{\varepsilon}\left(x-1\right)\quad\mbox{for }x\in(1,1+\varepsilon),
\]
see Fig.~\ref{fig:fk}. 
\begin{figure}[b]
\centering{}\includegraphics[width=8cm]{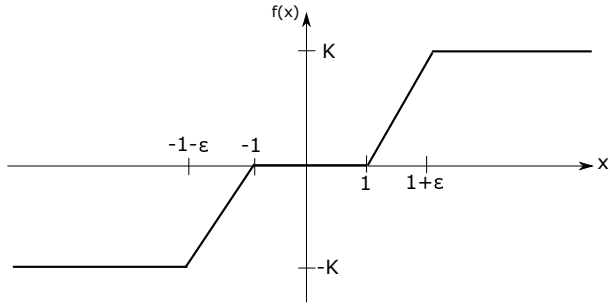} \caption{The plot of $f_{K}$.}
\label{fig:fk} 
\end{figure}

The results of the paper are expected to hold if the nondecreasing
continuous function $f_{K}$ is defined differently on $(-1-\varepsilon,-1)\cup(1,1+\varepsilon)$,
or if the coefficient of the linear term on the right hand side of
\eqref{Eq} is $-\mu$ with $\mu>0$.

The phase space for \eqref{Eq} is the Banach space $C=C\left(\left[-1,0\right],\mathbb{R}\right)$
with the maximum norm. If for some $t\in\mathbb{R}$, the interval
$\left[t-1,t\right]$ is in the domain of a continuous function $x$,
then the segment $x_{t}\in C$ is defined by $x_{t}\left(s\right)=x\left(t+s\right)$
for $-1\leq s\leq0$. 

If $K>1+\varepsilon$, then $f_{K}$ has two fixed points $\chi_{-}\in(-1-\varepsilon,-1)$
and $\chi_{+}\in(1,1+\varepsilon)$ with $f_{K}'\left(\chi_{-}\right)>1$
and $f_{K}'\left(\chi_{+}\right)>1$. Thus the constant elements 
\[
\left[-1,0\right]\ni s\mapsto\chi_{-}\in\mathbb{R}\mbox{ and }\left[-1,0\right]\ni s\mapsto\chi_{+}\in\mathbb{R}
\]
of $C$ are unstable equilibria. We know that there exist periodic
solutions oscillating about either $\chi_{-}$ or $\chi_{+}$ if $K$
is sufficiently large, and these periodic orbits appear via Hopf-bifurcations
\cite{Vas1}. 

We say that a periodic solution has large amplitude if it oscillates
about both $\chi_{-}$ and $\chi_{+}$. The corresponding orbit is
a large-amplitude periodic orbit. The existence of a pair of large-amplitude
periodic orbits has been first shown in \cite{Krisztin-Vas} for a
similar nonlinearity $f_{K}$ with $K$ large enough. More complicated
configurations of such periodic orbits has appeared in \cite{Vas2}.
A third work in this topic, the paper \cite{Krisztin-Vas_2} has described
the complicated geometric structure of the unstable set of a large-amplitude
periodic orbit in detail. These works have not explained how these
periodic orbits bifurcate as the parameter $K$ changes. Apparently
they cannot appear via Hopf bifurcation in a neighborhood of an unstable
equilibrium. In this paper we verify that for the nonlinearity $f_{K}$
defined above, large-amplitude periodic orbits arise via a saddle-node
bifurcation. The following theorem has already appeared in \cite{Krisztin-Vas}
as a conjecture. 
\begin{thm}
\label{main thm}(Saddle-node bifurcation of periodic orbits) For
all sufficiently small positive $\varepsilon$, one can give a threshold
parameter $K^{*}=K^{*}\left(\varepsilon\right)\in\left(6.5,7\right)$,
a large-amplitude periodic solution $p=p\left(\varepsilon\right)\colon\mathbb{R}\rightarrow\mathbb{R}$
of (\ref{Eq}) for the parameter $K=K^{*}$, an open neighborhood
$B=B(\varepsilon)$ of its initial segment $p_{0}$ in $C$ and a
constant $\delta=\delta(\varepsilon)>0$ such that

(i) if $K\in\left(K^{*}-\delta,K^{*}\right)$, then no periodic orbit
for (\ref{Eq}) has segments in $B$;

(ii) if $K=K^{*}$, then $\mathcal{O}=\left\{ p_{t}:t\in\mathbb{R}\right\} $
is the only periodic orbit with segments in $B$;

(iii) if $K\in\left(K^{*},K^{*}+\delta\right)$, then there are exactly
two periodic orbits with segments in $B$, and both of them are of
large-amplitude. 
\end{thm}
Let $K_{0}$ be that solution of equation 
\begin{equation}
\left(K-1\right)\left(K+1\right)^{3}=e\left(K^{2}-2K-1\right)^{2}\label{ke}
\end{equation}
that belongs to the interval $\left(6.5,7\right)$. It is easy to
show that $K_{0}$ is unique, see Section 3 of \cite{Krisztin-Vas}.
Numerical computation shows that $K_{0}\approx6.87$. We will see
that the limit of the bifurcation parameter $K^{*}(\varepsilon)$
is $K_{0}$ as $\varepsilon\rightarrow0^{+}$.

The proof is organized as follows. Let $\varepsilon\in\left(0,1\right)$.
In Section 2 we introduce a one-dimensional map $F$ depending also
on parameters $K$ and $\varepsilon$. In Section 3 we show that the
fixed points of $F(\cdot,K,\varepsilon)$ determine large-amplitude
periodic solutions for equation (\ref{Eq}). Then we show in Section
4 that $F$ undergoes a saddle-node bifurcation as $K$ varies if
$\varepsilon$ is a fixed and sufficiently small positive number.
We also need to show that -- locally -- all periodic solutions can
be obtained as fixed points of $F(\cdot,K,\varepsilon)$. This is
done Section 5. Theorem \ref{main thm} immediately follows from these
results, see Section 6. The Appendix contains certain lengthy but
straightforward calculations used in Section 4.

In the saddle-node bifurcation of $F$, a neutral fixed point splits
into two fixed points, one attracting and one repelling. This does
not imply that we have one stable and one unstable periodic orbit
for $K>K^{*}$. We know that if $f_{K}$ is a $C^{1}$-function with
nonnegative derivative, then all periodic orbits 
are unstable, see e.g., Proposition 7.1 in \cite{Vas2}. Hence we
presume that the periodic orbits given by the above theorem are also
unstable. 

In the previous paper \cite{Krisztin-Vas}, we obtained large-amplitude
periodic solutions also as fixed points of finite dimensional maps.
We emphasize that here we construct $F$ in a different way. The current
approach is simpler because it yields shorter calculations. The advantage
of the construction used in \cite{Krisztin-Vas} is the following: the eigenvalues
of the derivatives of the finite dimensional maps in \cite{Krisztin-Vas} at the fixed points
coincide with the Floquet multipliers of the corresponding periodic
orbits. Hence those finite dimensional maps 
give precise information on the stability properties of the periodic
orbits. This is not true here. 

The reader may find other examples,
in which the existence of a periodic orbit is shown by handling a
finite dimensional fixed point problem, e.g., in the papers \cite{Ivanov-Losson,Lani-Wayda,Stoffer}.

\section{The map $F$}

Let $\varepsilon\in(0,1)$ and $K\in\left(6.5,7\right)$. In this
section we define a periodic function $p$ as the concatenation of
certain auxiliary functions $y_{1},y_{2},...,y_{10}$ such that if
$p$ is a solution of the delay equation \eqref{Eq}, then $y_{1},y_{2},...,y_{10}$
satisfy a system of ordinary differential equations with boundary
conditions. Then we reduce this ODE system to a single fixed point
equation of the form $F\left(L_{2},K,\varepsilon\right)=L_{2}$, where
$L_{2}$ is a parameter corresponding to $p$. 

Assume that 
\begin{lyxlist}{00.00.0000}
\item [{(H1)}] $L_{i}>0$ for $i\in\{1,2,...,5\}$, 
\item [{(H2)}] $2L_{1}+5L_{2}+5L_{3}+3L_{4}+3L_{5}=1,$
\item [{(H3)}] $\theta_{i}>1+\varepsilon$ for $i\in\{1,2,3,4\}$, and
$\theta_{i}\in(1,1+\varepsilon)$ for $i\in\{5,6\}$.
\end{lyxlist}
Consider the subsequent continuous functions: 
\begin{lyxlist}{00.00.0000}
\item [{(H4)}] $y_{1}\in C([0,L_{1}],\mathbb{R})$ with $y_{1}(0)=1+\varepsilon$
and $y_{1}(L_{1})=\theta_{1}$,\\
 $y_{2}\in C([0,L_{2}],\mathbb{R})$ with $y_{2}(0)=\theta_{1}$ and
$y_{2}(L_{2})=\theta_{2}$,\\
 $y_{3}\in C([0,L_{3}],\mathbb{R})$ with $y_{3}(0)=\theta_{2}$ and
$y_{3}(L_{3})=\theta_{3}$,\\
 $y_{4}\in C([0,L_{4}],\mathbb{R})$ with $y_{4}(0)=\theta_{3}$ and
$y_{4}(L_{4})=\theta_{4}$,\\
 $y_{5}\in C([0,L_{5}],\mathbb{R})$ with $y_{5}(0)=\theta_{4}$ and
$y_{5}(L_{5})=1+\varepsilon$,\\
 $y_{6}\in C([0,L_{2}],\mathbb{R})$ with $y_{6}(0)=1+\varepsilon$
and $y_{6}(L_{2})=\theta_{5}$,\\
 $y_{7}\in C([0,L_{3}],\mathbb{R})$ with $y_{7}(0)=\theta_{5}$ and
$y_{7}(L_{3})=\theta_{6}$,\\
 $y_{8}\in C([0,L_{4}],\mathbb{R})$ with $y_{8}(0)=\theta_{6}$ and
$y_{8}(L_{4})=1$,\\
 $y_{9}\in C([0,L_{2}+L_{5}],\mathbb{R})$ with $y_{9}(0)=1$ and
$y_{9}(L_{2}+L_{5})=-1$,\\
 $y_{10}\in C([0,L_{3}],\mathbb{R})$ with $y_{10}(0)=-1$ and $y_{10}(L_{3})=-1-\varepsilon$,
\item [{(H5)}] if $i\in\{1,2,...,5\}$, then $y_{i}(s)>1+\varepsilon$
for all $s$ in the interior of the domain of $y_{i},$\\
if $i\in\{6,7,8\}$, then $y_{i}(s)\in(1,1+\varepsilon)$ for all
$s$ in the interior of the domain of $y_{i}$,\\
$y_{9}(s)\in(-1,1)$ for all $s\in(0,L_{2}+L_{5})$,\\
$y_{10}(s)\in(-1-\varepsilon,-1)$ for all $s\in(0,L_{3})$. 
\end{lyxlist}
Fig.~\ref{fig:periodic function} plots certain horizontal translations
of $y_{1},...,y_{10}.$

\begin{figure}[b]
\centering{}\centering \includegraphics[angle=90,width=0.75\textwidth]{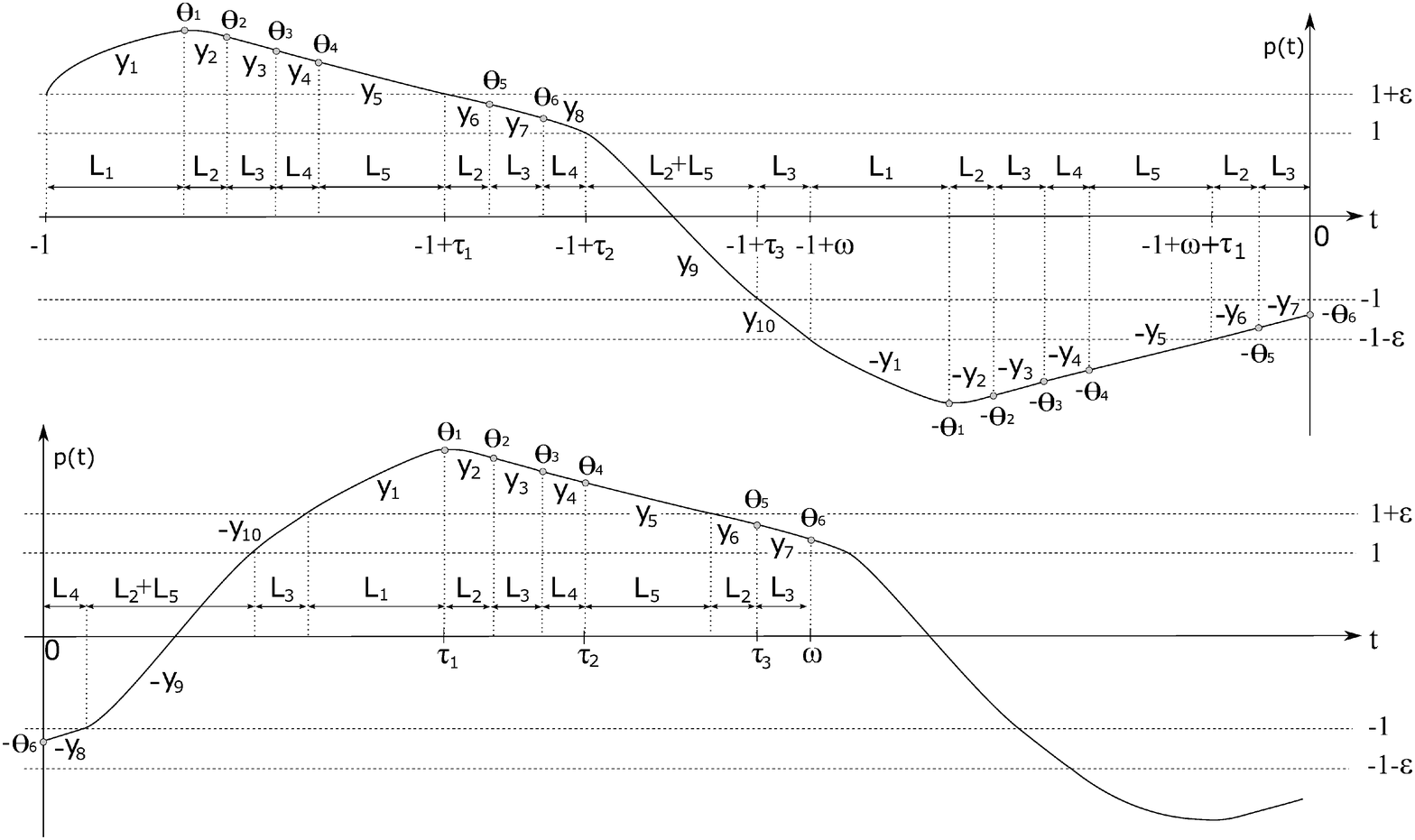}
\caption{The plot of p on {[}-1,0{]} and on {[}0,1{]}.}
\label{fig:periodic function} 
\end{figure}

Set $0<\tau_{1}<\tau_{2}<\tau_{3}<\omega<1$ as 
\begin{align}
\tau_{1} & =\sum_{i=1}^{5}{L_{i}},\nonumber \\
\tau_{2} & =\tau_{1}+L_{2}+L_{3}+L_{4},\label{tau  omega}\\
\tau_{3} & =\tau_{2}+L_{2}+L_{5},\nonumber \\
\omega & =\tau_{3}+L_{3}.\nonumber 
\end{align}
Introduce a $2\omega$-periodic function $p:\mathbb{R}\rightarrow\mathbb{R}$
as follows. Set $p$ on $[-1,-1+\omega]$ such that 
\begin{align}
p(t-1) & =y_{1}(t)\quad\mbox{for }t\in[0,L_{1}],\notag\label{p:egyenletek}\\
p(t-1+L_{1}) & =y_{2}(t)\quad\mbox{for }t\in[0,L_{2}],\notag\\
p(t-1+L_{1}+L_{2}) & =y_{3}(t)\quad\mbox{for }t\in[0,L_{3}],\notag\\
p(t-1+L_{1}+L_{2}+L_{3}) & =y_{4}(t)\quad\mbox{for }t\in[0,L_{4}],\notag\\
p(t-1+L_{1}+L_{2}+L_{3}+L_{4}) & =y_{5}(t)\quad\mbox{for }t\in[0,L_{5}],\tag{P.1}\\
p(t-1+\tau_{1}) & =y_{6}(t)\quad\mbox{for }t\in[0,L_{2}],\notag\\
p(t-1+\tau_{1}+L_{2}) & =y_{7}(t)\quad\mbox{for }t\in[0,L_{3}],\notag\\
p(t-1+\tau_{1}+L_{2}+L_{3}) & =y_{8}(t)\quad\mbox{for }t\in[0,L_{4}],\notag\\
p(t-1+\tau_{2}) & =y_{9}(t)\quad\mbox{for }t\in[0,L_{2}+L_{5}],\notag\\
p(t-1+\tau_{3}) & =y_{10}(t)\quad\mbox{for }t\in[0,L_{3}].\notag
\end{align}

Let 
\begin{equation}
p(t)=-p(t-\omega)\quad\mbox{for all }t\in[-1+\omega,-1+2\omega].\tag{P.2}\label{p:periodikus}
\end{equation}

Then extend $p$ to the real line $2\omega$-periodically. See Fig.\ \ref{fig:periodic function}
for the plot of $p$ on $\left[-1,1\right]$. It is clear that $p$
is of large amplitude. 

Our first goal is to find what conditions hold for $L_{1},...,L_{5},\theta_{1},...,\theta_{6}$
and $y_{1},...,y_{10}$ if $p$ satisfies equation (\ref{Eq}) for
all $t\in\mathbb{R}$. As $p(t)=-p(t-\omega)$ for all real $t$ and
$f_{K}$ is odd, we do not lose information if we restrict our examinations
to the interval $[0,\omega]$. So consider 
\begin{equation}
\dot{p}(t)=-p(t)+f_{K}(p(t-1))\quad\mbox{for }t\in[0,\omega].\label{eq:1}
\end{equation}
We study \eqref{eq:1} first on the interval $[0,\tau_{1}]$, then
on $[\tau_{1},\tau_{2}]$, $[\tau_{2},\tau_{3}]$ and $[\tau_{3},\omega]$.

\emph{1.} \emph{The interval $[0,\tau_{1}].$ }The way we extended
$p$ from $[-1,-1+\omega]$ to $\mathbb{R}$ and condition (H2) together
imply that 
\begin{align}
p(t) & =-y_{8}(t)\quad\mbox{for }t\in[0,L_{4}],\notag\nonumber \\
p(t+L_{4}) & =-y_{9}(t)\quad\mbox{for }t\in[0,L_{2}+L_{5}],\label{p on (0,tau1)}\\
p(t+L_{4}+L_{2}+L_{5}) & =-y_{10}(t)\quad\mbox{for }t\in[0,L_{3}],\notag\nonumber \\
p(t+L_{4}+L_{2}+L_{5}+L_{3}) & =y_{1}(t)\quad\mbox{for }t\in[0,L_{1}],\notag\nonumber 
\end{align}
see Fig.~\ref{fig:periodic function}. Also observe -- using \eqref{p:egyenletek} and
(H3)-(H5) -- that 
\[
p(t)\geq1+\varepsilon\quad\mbox{for }t\in[-1,-1+\tau_{1}],
\]
and thus (\ref{eq:1}) is in the form $\dot{p}\left(t\right)=-p\left(t\right)+K$
on $\left[0,\tau_{1}\right]$. We conclude that (\ref{eq:1}) holds
for $t\in[0,\tau_{1}]$ if and only if the subsequent four equations
are satisfied:
\begin{align}
\dot{y}_{8}(t)= & -y_{8}(t)-K,\quad t\in[0,L_{4}],\label{eq:y8}\\
\dot{y}_{9}(t)= & -y_{9}(t)-K,\quad t\in[0,L_{2}+L_{5}],\label{eq:y9}\\
\dot{y}_{10}(t)= & -y_{10}(t)-K,\quad t\in[0,L_{3}],\label{eq:y10}\\
\dot{y}_{1}(t)= & -y_{1}(t)+K,\quad t\in[0,L_{1}].\label{eq:y1}
\end{align}

\emph{2.} \emph{The interval $[\tau_{1},\tau_{2}].$ }By the definition
of $p$ and hypothesis (H2), 
\begin{align*}
p(t+\tau_{1}) & =y_{2}(t)\quad\mbox{for }t\in[0,L_{2}],\\
p(t+\tau_{1}+L_{2}) & =y_{3}(t)\quad\mbox{for }t\in[0,L_{3}]
\end{align*}
and 
\[
p(t+\tau_{1}+L_{2}+L_{3})=y_{4}(t)\quad\mbox{for }t\in[0,L_{4}].
\]
We also know from (\ref{p:egyenletek}) that 
\begin{align*}
p(t-1+\tau_{1}) & =y_{6}(t)\quad\mbox{for }t\in[0,L_{2}],\\
p(t-1+\tau_{1}+L_{2}) & =y_{7}(t)\quad\mbox{for }t\in[0,L_{3}],\\
p(t-1+\tau_{1}+L_{2}+L_{3}) & =y_{8}(t)\quad\mbox{for }t\in[0,L_{4}].
\end{align*}
Hypotheses (H3)-(H5) then guarantee that 
\[
p(t)\in[1,1+\varepsilon]\quad\mbox{for }t\in[-1+\tau_{1},-1+\tau_{2}].
\]
Using the definition of $f_{K}$, we obtain that (\ref{eq:1}) holds
on $[\tau_{1},\tau_{2}]$ if and only if
\begin{align}
\dot{y}_{2}(t)= & -y_{2}(t)+\frac{K}{\varepsilon}(y_{6}(t)-1)\quad\mbox{for }t\in[0,L_{2}],\label{eq:y2}\\
\dot{y}_{3}(t)= & -y_{3}(t)+\frac{K}{\varepsilon}(y_{7}(t)-1)\quad\mbox{for }t\in[0,L_{3}]\label{eq:y3}
\end{align}
and 
\begin{equation}
\dot{y}_{4}(t)=-y_{4}(t)+\frac{K}{\varepsilon}(y_{8}(t)-1)\quad\mbox{for }t\in[0,L_{4}].\label{eq:y4}
\end{equation}

\emph{3. The interval $[\tau_{2},\tau_{3}].$} Next observe that 
\begin{align*}
p(t+\tau_{2}) & =y_{5}(t)\quad\mbox{for }t\in[0,L_{5}],\\
p(t+\tau_{2}+L_{5}) & =y_{6}(t)\quad\mbox{for }t\in[0,L_{2}]
\end{align*}
and 
\[
p(t)\in[-1,1]\quad\mbox{for }t\in[-1+\tau_{2},-1+\tau_{3}].
\]
Therefore 
\begin{equation}
\dot{y}_{5}(t)=-y_{5}(t)\quad\mbox{for }t\in[0,L_{5}]\label{eq:y5}
\end{equation}
and 
\begin{equation}
\dot{y}_{6}(t)=-y_{6}(t)\quad\mbox{for }t\in[0,L_{2}].\label{eq:y6}
\end{equation}

\emph{4. The interval $[\tau_{3},\omega].$} At least observe that
\begin{align*}
p(t+\tau_{3}) & =y_{7}(t)\quad\mbox{for }t\in[0,L_{3}],\\
p(t-1+\tau_{3}) & =y_{10}(t)\quad\mbox{for }t\in[0,L_{3}],
\end{align*}
 and 
\[
p(t)\in[-1-\varepsilon,-1]\quad\mbox{for }t\in[-1+\tau_{3},-1+\omega].
\]
So on the interval $[\tau_{3},\omega]$, equation (\ref{eq:1}) is
equivalent to 
\begin{equation}
\dot{y}_{7}(t)=-y_{7}(t)+\frac{K}{\varepsilon}(y_{10}(t)+1),\quad t\in[0,L_{3}].\label{eq:y7}
\end{equation}

We see that under hypotheses (H1)-(H5), equation (\ref{eq:1}) is equivalent to a system
of linear ordinary differential equations. It worth solving equations
(\ref{eq:y8})-(\ref{eq:y1}) and (\ref{eq:y5})-(\ref{eq:y6}) first
because they are independent from the other ones. Then we can solve
(\ref{eq:y2}), (\ref{eq:y4}) and (\ref{eq:y7}) using the solutions
of (\ref{eq:y6}), (\ref{eq:y8}) and (\ref{eq:y10}), respectively.
At last, using the solution of (\ref{eq:y7}), we can find the solution
of (\ref{eq:y3}). Applying the boundary conditions given by (H4)
for $t=0$, we obtain that 
\begin{align}
y_{1}(t) & =K-(K-1-\varepsilon)e^{-t},\ t\in[0,L_{1}],\tag{Y.1}\label{eqsol:y1}\\
y_{2}(t) & =\theta_{1}e^{-t}+\frac{K}{\varepsilon}\left((1+\varepsilon)te^{-t}+e^{-t}-1\right),\ t\in[0,L_{2}],\tag{Y.2}\label{eqsol:y2}\\
y_{3}(t) & =\theta_{2}e^{-t}+\frac{K}{\varepsilon}\left((\theta_{5}t+1)e^{-t}-1\right)\quad\tag{Y.3}\label{eqsol:y3}\\
 & \quad-\frac{K^{2}}{\varepsilon^{2}}(K-1)\left(1-\left(1+t+\frac{t^{2}}{2}\right)e^{-t}\right),\ t\in[0,L_{3}],\notag\\
y_{4}(t) & =\theta_{3}e^{-t}+\frac{K}{\varepsilon}\left((K+\theta_{6})te^{-t}-(K+1)\left(1-e^{-t}\right)\right),\ t\in[0,L_{4}],\tag{Y.4}\label{eqsol:y4}\\
y_{5}(t) & =\theta_{4}e^{-t},\ t\in[0,L_{5}],\tag{Y.5}\label{eqsol:y5}\\
y_{6}(t) & =(1+\varepsilon)e^{-t},\ t\in[0,L_{2}],\tag{Y.6}\label{eqsol:y6}\\
y_{7}(t) & =\theta_{5}e^{-t}-\frac{K}{\varepsilon}(K-1)\left(1-(1+t)e^{-t}\right),\ t\in[0,L_{3}],\tag{Y.7}\label{eqsol:y7}\\
y_{8}(t) & =(K+\theta_{6})e^{-t}-K,\ t\in[0,L_{4}],\tag{Y.8}\label{eqsol:y8}\\
y_{9}(t) & =(K+1)e^{-t}-K,\ t\in[0,L_{2}+L_{5}],\tag{Y.9}\label{eqsol:y9}\\
y_{10}(t) & =(K-1)e^{-t}-K,\ t\in[0,L_{3}].\tag{Y.10}\label{eqsol:y10}
\end{align}
If we apply the boundary conditions given for the right end points
of the domains of $y_{i},$ $i\in\{1,...,10\}$, then we get the following
relations: 
\begin{align}
\theta_{1} & =K-(K-1-\varepsilon)e^{-L_{1}},\tag{B.1}\label{sol:bk1}\\
\theta_{2} & =\theta_{1}e^{-L_{2}}+\frac{K}{\varepsilon}\left((1+\varepsilon)L_{2}e^{-L_{2}}+e^{-L_{2}}-1\right),\tag{B.2}\label{sol:bk2}\\
\theta_{3} & =\theta_{2}e^{-L_{3}}+\frac{K}{\varepsilon}\left((\theta_{5}L_{3}+1)e^{-L_{3}}-1\right)\tag{B.3}\label{sol:bk3}\\
 & \quad-\frac{K^{2}}{\varepsilon^{2}}(K-1)\left(1-\left(1+L_{3}+\frac{L_{3}^{2}}{2}\right)e^{-L_{3}}\right),\notag\\
\theta_{4} & =\theta_{3}e^{-L_{4}}+\frac{K}{\varepsilon}\left((K+\theta_{6})L_{4}e^{-L_{4}}-(K+1)\left(1-e^{-L_{4}}\right)\right),\tag{B.4}\label{sol:bk4}\\
1+\varepsilon & =\theta_{4}e^{-L_{5}},\tag{B.5}\label{sol:bk5}\\
\theta_{5} & =(1+\varepsilon)e^{-L_{2}},\tag{B.6}\label{sol:bk6}\\
\theta_{6} & =\theta_{5}e^{-L_{3}}-\frac{K}{\varepsilon}(K-1)\left(1-(1+L_{3})e^{-L_{3}}\right),\tag{B.7}\label{sol:bk7}\\
1 & =(K+\theta_{6})e^{-L_{4}}-K,\tag{B.8}\label{sol:bk8}\\
-1 & =(K+1)e^{-L_{2}-L_{5}}-K,\tag{B.9}\label{sol:bk9}\\
-1-\varepsilon & =(K-1)e^{-L_{3}}-K.\tag{B.10}\label{sol:bk10}
\end{align}

Next we reduce the algebraic system of equations (H2), (\ref{sol:bk1})-(\ref{sol:bk10})
to a single equation for $L_{2},K$ and $\varepsilon$. Meanwhile,
we express $L_{1},L_{3},L_{4},L_{5}$ and $\theta_{1},\theta_{2},...,\theta_{6}$
as functions of $L_{2},K$ and $\varepsilon$. 

By (\ref{sol:bk10}), 
\begin{equation}
L_{3}=\ln{\frac{K-1}{K-1-\varepsilon}}.\tag{C.1}\label{L3}
\end{equation}
From (\ref{sol:bk9}) and (\ref{sol:bk5}) we obtain that 
\begin{equation}
L_{5}=\ln{\frac{K+1}{K-1}}-L_{2}\tag{C.2}\label{L5}
\end{equation}
and 
\begin{equation}
\theta_{4}=(1+\varepsilon)\frac{K+1}{K-1}e^{-L_{2}}.\tag{C.3}\label{theta4}
\end{equation}
$\theta_{5}$ is already expressed in (\ref{sol:bk6}). In order to
simplify reference to the formulas in this section, we repeat that
\begin{equation}
\theta_{5}=(1+\varepsilon)e^{-L_{2}}.\tag{C.4}\label{theta5}
\end{equation}
 Using this, (\ref{sol:bk7}) and (\ref{L3}), we calculate that 
\begin{equation}
\theta_{6}=(1+\varepsilon)\frac{K-1-\varepsilon}{K-1}e^{-L_{2}}+\frac{K}{\varepsilon}(K-1-\varepsilon)\ln{\frac{K-1}{K-1-\varepsilon}}-K.\tag{C.5}\label{theta6}
\end{equation}
Note that $\theta_{6}+K>0$. From (\ref{sol:bk8}) we obtain that
\begin{equation}
L_{4}=\ln{\frac{K+\theta_{6}}{K+1}}.\tag{C.6}\label{L4}
\end{equation}
Now we use (H2), (\ref{L3}) and (\ref{L4}) to express $L_{1}$:
\begin{equation}
L_{1}=\frac{1}{2}-L_{2}+\frac{5}{2}\ln{(K-1-\varepsilon)}-\frac{3}{2}\ln{\left(K+\theta_{6}\right)}-\ln{(K-1)}.\tag{C.7}\label{L1}
\end{equation}
Substituting the last relation into (\ref{sol:bk1}), we get that
\begin{equation}
\theta_{1}=K-\frac{e^{L_{2}-\frac{1}{2}}(K-1)\left(K+\theta_{6}\right)^{\frac{3}{2}}}{(K-1-\varepsilon)^{\frac{3}{2}}}.\tag{C.8}\label{theta1}
\end{equation}
Then replacing $\theta_{1}$ by (\ref{theta1}) in (\ref{sol:bk2}),
we conclude that 
\begin{equation}
\theta_{2}=Ke^{-L_{2}}-\frac{e^{-\frac{1}{2}}(K-1)(K+\theta_{6})^{\frac{3}{2}}}{(K-1-\varepsilon)^{\frac{3}{2}}}+\frac{K}{\varepsilon}\left((1+\varepsilon)L_{2}e^{-L_{2}}+e^{-L_{2}}-1\right).\tag{C.9}\label{theta2}
\end{equation}
Parameter $\theta_{3}$ appeared as a function of $K,\varepsilon,\theta_{2},\theta_{5}$
and $L_{3}$ in (\ref{sol:bk3}). As $\theta_{2},\theta_{5}$ and
$L_{3}$ have already been given as functions of $L_{2},K$ and $\varepsilon$,
now we see that $\theta_{3}$ can also be expressed as a function
of $L_{2},K$ and $\varepsilon$. We will consider $\theta_{3}$ in
the form
\begin{alignat}{1}
\theta_{3} & =\theta_{2}e^{-L_{3}}+\frac{K}{\varepsilon}\left((1+\varepsilon)L_{3}e^{-L_{2}-L_{3}}+e^{-L_{3}}-1\right)\tag{C.10}\label{theta3}\\
 & \quad-\frac{K^{2}}{\varepsilon^{2}}(K-1)\left(1-\left(1+L_{3}+\frac{L_{3}^{2}}{2}\right)e^{-L_{3}}\right),\notag
\end{alignat}
where $\theta_{2}$ and $L_{3}$ are defined by (\ref{theta2}) and
(\ref{L3}), respectively.

Then (\ref{sol:bk4}) is the only algebraic equation we have not used
so far. We substitute (\ref{theta4}) into the left hand side of (\ref{sol:bk4})
and then multiply this equation by $e^{L_{4}}=(K+\theta_{6})/(K+1)$.
We deduce that 
\[
(1+\varepsilon)\frac{K+\theta_{6}}{K-1}e^{-L_{2}}=\frac{K}{\varepsilon}(K+1)\left(1-(1-L_{4})e^{L_{4}}\right)+\theta_{3}.
\]

Let 
\[
U=\left\{ (L_{2},K,\varepsilon)\in\mathbb{R}^{3}\colon\ \varepsilon\in(0,1),\ K\in(6.5,7),\ L_{2}\in(-\varepsilon,\varepsilon)\right\} .
\]
If we consider $\theta_3,\theta_6$ and $L_4$ as functions on $U$ given by (\ref{theta3}), (\ref{theta6}) and (\ref{L4}), then we can define a map $F:U\rightarrow\mathbb{R}$ by 
\[
F(L_{2},K,\varepsilon)=\frac{K}{\varepsilon}(K+1)\left(1-(1-L_{4})e^{L_{4}}\right)+\theta_{3}-(1+\varepsilon)\frac{K+\theta_{6}}{K-1}e^{-L_{2}}+L_{2}
\]
for all $(L_{2},K,\varepsilon)\in U$. One can easily check that $F$ is well-defined and continuous on $U$. 

The following proposition holds.
\begin{prop}
\label{prop:per sol fixed point}Let $\varepsilon\in(0,1)$ and $K\in\left(6.5,7\right)$.
Suppose that $p\colon\mathbb{R}\rightarrow\mathbb{R}$ is a $2\omega$-periodic
solution of \eqref{Eq}, $p$ is the concatenation of functions $y_{1},y_{2},\ldots,y_{10}$
as given in \eqref{p:egyenletek}-\eqref{p:periodikus}, furthermore
the functions $y_{1},y_{2},\ldots,y_{10}$ satisfy (H1)-(H5) with
some parameters $L_{i}>0$, $i\in\{1,2,...,5\}$, and $\theta_{i}$,
$i\in\left\{ 1,\ldots,6\right\} $. Then $L_{2}\in\left(0,\varepsilon\right)$
and $F(L_{2},K,\varepsilon)=L_{2}$. \end{prop}
\begin{proof}
Recall from \eqref{theta5} that $\theta_{5}=\left(1+\varepsilon\right)e^{-L_{2}}$,
which is greater than $1$ by (H3). It follows immediately that $L_{2}<\ln\left(1+\varepsilon\right)<\varepsilon$.
The rest of the statement comes from the above calculations. 
\end{proof}
We need to consider $F$ also for $L_{2}\in(-\varepsilon,0]$  because of technical reasons;
see Proposition \ref{prop:ift} in Section 4. We will also use the following remark in the next section.
\begin{rem}
\label{equivalency}The system of algebraic equations $F(L_{2},K,\varepsilon)=L_{2}$,
(\ref{L3})-(\ref{theta3}) is equivalent to the system of equations
(H2), (\ref{sol:bk1})-(\ref{sol:bk10}). 
\end{rem}

\section{The fixed points of $F$ yield periodic solutions}

By the previous section, if (H1)-(H5) hold, and $p\colon\mathbb{R}\rightarrow\mathbb{R}$
is a $2\omega$-periodic solution of \eqref{Eq} given by \eqref{p:egyenletek}-\eqref{p:periodikus},
then $L_{2}\mapsto F(L_{2},K,\varepsilon)$ has a fixed point. We
devote this section to verify the converse statement: if $\varepsilon>0$
is small enough and $K\in\left(6.5,7\right)$, then all sufficiently
small positive fixed points of $L_{2}\mapsto F(L_{2},K,\varepsilon)$
yield periodic solutions of \eqref{Eq}.

We need to consider $L_{1},L_{3},L_{4},L_{5}$ and $\theta_{i},$
$1\leq i\leq6$, as functions of $L_{2},K$ and $\varepsilon$ (and
not as parameters given by hypotheses (H1)-(H5)). So assume that 
\begin{lyxlist}{00.00.0000}
\item [{(H6)}] $L_{i}$, $i\in\{1,3,4,5\}$, and $\theta_{i},$ $1\leq i\leq6$,
are defined by (\ref{L3})-(\ref{theta3}) on 
\[
U=\left\{ (L_{2},K,\varepsilon)\in\mathbb{R}^{3}\colon\ \varepsilon\in(0,1),\ K\in(6.5,7)\ \mbox{and }L_{2}\in(-\varepsilon,\varepsilon)\right\} .
\]

\end{lyxlist}
One easily check that $L_{i}$, $i\in\{1,3,4,5\}$, and $\theta_{i},$
$1\leq i\leq6$, are continuous functions of $(L_{2},K,\varepsilon)$
on $U$.

In this section we also need the assumption that
\begin{lyxlist}{00.00.0000}
\item [{(H7)}] $y_{1},...,y_{10}$ are the solutions (\ref{eqsol:y1})-(\ref{eqsol:y10})
of the ordinary differential equations (\ref{eq:y8})-(\ref{eq:y7}).
\end{lyxlist}
Set 
\[
\theta^{*}\left(\bar{K}\right)=\bar{K}-\sqrt{\frac{\left(\bar{K}+1\right)^{3}}{e\left(\bar{K}-1\right)}}\qquad\mbox{for }\bar{K}\in\left[6.5,7\right].
\]
We claim that $\theta^{*}\left(\bar{K}\right)>1$ for $\bar{K}\in\left[6.5,7\right]$.
As this inequality is equivalent to 
\[
\left(\bar{K}-1\right)\left(1-\sqrt{\frac{\left(\bar{K}+1\right)^{3}}{e\left(\bar{K}-1\right)^{3}}}\right)>0,
\]
 we need to verify that $\left(\bar{K}+1\right)^{3}/\left(\bar{K}-1\right)^{3}<e$ holds.
Indeed, since $\bar{K}\mapsto\left(\bar{K}+1\right)/\left(\bar{K}-1\right)$
is strictly decreasing for $\bar{K}>1$, we see that 
\begin{equation}
\left(\frac{\bar{K}+1}{\bar{K}-1}\right)^{3}\leq\left(\frac{6.5+1}{6.5-1}\right)^{3}=\left(\frac{15}{11}\right)^{3}=2+\frac{713}{1331}\leq2+\frac{800}{1200}=2+\frac{2}{3}<e\label{estimate}
\end{equation}
for $\bar{K}\in\left[6.5,7\right]$.

The first two statements of the subsequent proposition give information
on the behavior of $F$ for small positive $\varepsilon$. The third
statement examines the limit of $y_{2}(t)$, $y_{3}(t)$ and $y_{4}(t)$
for all $t$ in their domains as $\varepsilon\rightarrow0^{+}$. Since
$y_{2}$, $y_{3}$ and $y_{4}$ are well-defined by (\ref{eqsol:y2})-(\ref{eqsol:y4})
only if $L_{i}\geq0$ for $i\in\left\{ 2,3,4\right\} $, here we assume
that $L_{2}\geq0$ and $L_{4}\geq0$. It is clear that $L_{3}=\ln(K-1)-\ln(K-1-\varepsilon)$
is positive.
\begin{prop}
\label{prop:Fsorfejtes} The subsequent assertions hold under hypothesis
(H6).\\
 (i) $\theta_{6}=1+O(\varepsilon)$, $L_{4}=O(\varepsilon)$ and thus
\[
\frac{K}{\varepsilon}(K+1)\left(1-(1-L_{4})e^{L_{4}}\right)=O(\varepsilon)\quad\mbox{as }\varepsilon\rightarrow0^{+}.
\]
(ii) If $K\rightarrow\bar{K}\in[6.5,7]$ and $\varepsilon\rightarrow0^{+}$,
then $\theta_{3}$ converges to $\theta^{*}\left(\bar{K}\right)$.\\
(iii) Assume in addition that $L_{2}\geq0$ and $L_{4}\geq0$. Define
$y_{2}$, $y_{3}$ and $y_{4}$ by (\ref{eqsol:y2})-(\ref{eqsol:y4}).
If $K\rightarrow\bar{K}\in[6.5,7]$ and $\varepsilon\rightarrow0^{+}$,
then $y_{2}\left(t\right),y_{3}(t)$ and $y_{4}(t)$ converges to
$\theta^{*}\left(\bar{K}\right)$, uniformly in $t\in[0,L_{2}]$,
$t\in[0,L_{3}]$ and $t\in[0,L_{4}]$, respectively. 
\end{prop}
Before giving the proof of this proposition, let us make a remark
on notation $O$. If $g$ is a function of $L_{2},K,\varepsilon,t$
(or only some of these variables) on a set $D$, and $k$ is a positive
integer, then the expression $g=O\left(\varepsilon^{k}\right)$ as
$\varepsilon\rightarrow0^{+}$ (or simply $g=O\left(\varepsilon^{k}\right)$)
means that there exists $M>0$\emph{ }such that $|g\left(L_{2},K,\varepsilon,t\right)|\leq M\varepsilon^{k}$
if $\left(L_{2},K,\varepsilon,t\right)\in D$ and $\varepsilon>0$
is sufficiently close to zero. Constant $M$ is always independent
from\emph{ $L_{2},K$} and $t$ in this paper.
\begin{proof}
\textit{The proof of statement (i).} It is well-known that 
\begin{equation}
\ln(1+x)=x+O\left(x^{2}\right)\mbox{ as }x\rightarrow0.\label{expansion for ln}
\end{equation}
If $K\in\left(6.5,7\right)$ and $\varepsilon\rightarrow0^{+}$, then
\[
\frac{\varepsilon}{K-1-\varepsilon}\rightarrow0^{+}
\]
and thus 
\begin{equation}
\ln{\frac{K-1}{K-1-\varepsilon}}=\ln{\left(1+\frac{\varepsilon}{K-1-\varepsilon}\right)}=\frac{\varepsilon}{K-1-\varepsilon}+O\left(\varepsilon^{2}\right)\mbox{ as }\varepsilon\rightarrow0^{+}.\label{sorfejtesln}
\end{equation}
Therefore 
\begin{equation}
\frac{K}{\varepsilon}(K-1-\varepsilon)\ln{\frac{K-1}{K-1-\varepsilon}}=K+O(\varepsilon).\label{biz1}
\end{equation}
In addition, since $L_{2}\in(-\varepsilon,\varepsilon)$, 
\begin{equation}
(1+\varepsilon)\frac{K-1-\varepsilon}{K-1}e^{-L_{2}}=(1+\varepsilon)\left(1-\frac{\varepsilon}{K-1}\right)\left(1+O(L_{2})\right)=1+O(\varepsilon).\label{biz2}
\end{equation}
Substituting (\ref{biz1}) and (\ref{biz2}) into (\ref{theta6}),
we obtain that $\theta_{6}=1+O(\varepsilon)$.

Using (\ref{L4}), the previous statement regarding $\theta_{6}$
and \eqref{expansion for ln}, we immediately get that 
\begin{equation}
L_{4}=\ln{\left(1+\frac{\theta_{6}-1}{K+1}\right)}=O(\varepsilon).\label{biz3}
\end{equation}
By the power series expansion of the exponential function, 
\begin{equation}
1-e^{L_{4}}(1-L_{4})=O\left(L_{4}^{2}\right)\mbox{ as }L_{4}\rightarrow0.\label{biz4}
\end{equation}
The last statement of Proposition \ref{prop:Fsorfejtes}.(i) then
comes from (\ref{biz3}) and (\ref{biz4}). 

\emph{The proof of statement (iii).} Let us now prove\emph{ (iii)
}in three steps. Let $L_{2}\geq0$ and $L_{4}\geq0$.

1.\emph{ }The convergence of $y_{2}\left(t\right)$ for $t\in[0,L_{2}]$\emph{.}
We see from statement (i) and formula (\ref{theta1}) that 
\begin{equation}
\lim_{\substack{K\rightarrow\bar{K}\\
\varepsilon\rightarrow0^{+}
}
}\theta_{1}=\bar{K}-\sqrt{\frac{\left(\bar{K}+1\right)^{3}}{e\left(\bar{K}-1\right)}}=\theta^{*}\left(\bar{K}\right).\label{biz5}
\end{equation}
Using that $0\leq t\leq L_{2}<\varepsilon$ and $e^{x}=1+x+O\left(x^{2}\right)$
as $x\rightarrow0$, we also see that 
\begin{alignat}{1}
\frac{K}{\varepsilon}\left((1+\varepsilon)te^{-t}+e^{-t}-1\right) & =\frac{K}{\varepsilon}\left((1+\varepsilon)t\left(1-t+O\left(t^{2}\right)\right)-t+O\left(t^{2}\right)\right)=\label{biz6}\\
 & \quad\frac{K}{\varepsilon}\left(\varepsilon t+O\left(t^{2}\right)\right)=O(\varepsilon).\notag
\end{alignat}
Substituting (\ref{biz5}) and (\ref{biz6}) into (\ref{eqsol:y2}),
we get that $y_{2}(t)$ converges to $\theta^{*}\left(\bar{K}\right)$
for all $t\in[0,L_{2}],$ and this convergence is uniform in $t$. 

2. The convergence of $y_{3}(t)$ for $t\in[0,L_{3}]$, using formula
(\ref{eqsol:y3}). Observe that if $\theta_{1}$ is given by \eqref{theta1},
and $\theta_{2}$ is determined by \eqref{theta2}, then \eqref{eqsol:y2}
yields that $y_{2}(L_{2})=\theta_{2}$. So by our last result, $\lim_{\varepsilon\rightarrow0^{+},K\rightarrow\bar{K}}\theta_{2}=\theta^{*}\left(\bar{K}\right)$.
We also see from (\ref{theta5}) and from $L_{2}\in(-\varepsilon,\varepsilon)$
that $\theta_{5}=1+O(\varepsilon)\mbox{ as }\varepsilon\rightarrow0^{+}.$
Using this and the power series expansion of the exponential function,
we get the following for $0\leq t\leq L_{3}=O(\varepsilon)$: 
\begin{equation}
(\theta_{5}t+1)e^{-t}-1=((1+O(\varepsilon))t+1)\left(1-t+O\left(t^{2}\right)\right)-1=O\left(\varepsilon^{2}\right).\label{biz9}
\end{equation}
We also observe that 
\begin{equation}
1-e^{-t}\left(1+t+\frac{t^{2}}{2}\right)=1-\left(1-t+\frac{t^{2}}{2}+O(t^{3})\right)\left(1+t+\frac{t^{2}}{2}\right)=O\left(\varepsilon^{3}\right).\label{biz10}
\end{equation}
Summing up, (\ref{eqsol:y3}) yields that if $0\leq t\leq L_{3}=O(\varepsilon)$,
then 
\[
\lim_{\substack{K\rightarrow\bar{K}\\
\varepsilon\rightarrow0^{+}
}
}y_{3}(t)=\lim_{\substack{K\rightarrow\bar{K}\\
\varepsilon\rightarrow0^{+}
}
}\theta_{2}e^{-t}=\theta^{*}\left(\bar{K}\right).
\]
This convergence is uniform in $t$.

3. The convergence of $y_{4}\left(t\right)$ for $t\in\left[0,L_{4}\right]$.
On the one hand, if $y_{3}$ is defined by \eqref{eqsol:y3}, $\theta_{5}$
is defined (\ref{theta5}) and $\theta_{3}$ is given by \eqref{theta3},
then $\theta_{3}=y_{3}\left(L_{3}\right)$. Hence, by the previous
paragraph, $\theta_{3}$ converges to $\theta^{*}\left(\bar{K}\right)$
if $K\rightarrow\bar{K}\in[6.5,7]$ and $\varepsilon\rightarrow0^{+}$.
(Note that we have proved statement (ii) in the case $L_{2}\geq0$).
On the other hand, it follows from $\theta_{6}=1+O(\varepsilon)$
that 
\[
(K+\theta_{6})te^{-t}-(K+1)\left(1-e^{-t}\right)
\]
equals 
\[
(K+1+O(\varepsilon))\left(t+O\left(t^{2}\right)\right)-(K+1)\left(t+O\left(t^{2}\right)\right)=O\left(\varepsilon^{2}\right)
\]
for $0\leq t\leq L_{4}=O(\varepsilon)$. In consequence, formula (\ref{eqsol:y4})
shows that $y_{4}(t)$ converges to $\theta^{*}\left(\bar{K}\right)$,
uniformly in $t\in[0,L_{4}]$.

\emph{The proof of statement (ii).} We have already verified (ii)
for $L_{2}\in[0,\varepsilon)$. Now suppose that $L_{2}\in(-\varepsilon,0)$
and observe that \eqref{biz6} holds also with $t=L_{2}\in(-\varepsilon,0)$.
Therefore \eqref{theta2} and the equality $\theta_{6}=1+O(\varepsilon)$
together show that $\theta_{2}$ converges to $\theta^{*}\left(\bar{K}\right)$
also in the case when $L_{2}<0$. Now we can use $L_{3}=O(\varepsilon)$, \eqref{theta5} and \eqref{biz9}-\eqref{biz10} with $t=L_{3}$ to verify that $\theta_{3}$ (defined by \eqref{theta3}) converges
to $\theta^{*}\left(\bar{K}\right)$ if $K\rightarrow\bar{K}\in[6.5,7]$,
$\varepsilon\rightarrow0^{+}$ and $L_{2}\in(-\varepsilon,0)$.\end{proof}
\begin{cor}
\label{cor: limit of K} Assume that $\lim_{n\rightarrow\infty}\varepsilon_{n}=0^{+},$
\[
(L_{2,n},K_{n},\varepsilon_{n})\in U\quad\mbox{and}\quad F\left(L_{2,n},K_{n},\varepsilon_{n}\right)=L_{2,n}\mbox{ for all }n\geq0.
\]
Then $\left(K_{n}\right)_{n=0}^{\infty}$ is convergent, and $\lim_{n\rightarrow\infty}K_{n}=K_{0},$
where $K_{0}$ is the unique solution of (\ref{ke}) in {[}6.5,7{]}. \end{cor}
\begin{proof}
We already know from Section 3 of paper \cite{Krisztin-Vas} that
(\ref{ke}) has exactly one solution $K_{0}\approx6.87$ in $[6.5,7]$.

It suffices to prove that each subsequence of $\left(K_{n}\right)_{n=0}^{\infty}$
has a subsequence converging to $K_{0}$. As $K_{n}\in(6.5,7)$ for
all $n\geq1$, it is clear that any subsequence of $\left(K_{n}\right)_{n=0}^{\infty}$
has a convergent subsequence $\left(K_{n_{l}}\right)_{l=0}^{\infty}$.
Let 
\[
\bar{K}=\lim_{l\rightarrow\infty}K_{n_{l}}\in[6.5,7].
\]
Now let $l$ tend to infinity in the equation $F\left(L_{2,n_{l}},K_{n_{l}},\varepsilon_{n_{l}}\right)=L_{2,n_{l}}$.
Under the assumptions of the Corollary, $\lim_{l\rightarrow\infty}L_{2,n_{l}}=0$.
This fact, the definition of $F$ and Proposition \ref{prop:Fsorfejtes}.(i)-(ii)
together show that $\bar{K}\in[6.5,7]$ is a solution of 
\[
K-\sqrt{\frac{(K+1)^{3}}{e(K-1)}}-\frac{K+1}{K-1}=0.
\]
It is straightforward to show that this equation is equivalent to
(\ref{ke}), and thus $\bar{K}=K_{0}$. The proof is complete.
\end{proof}
For $K\in(6.5,7)$ and $\varepsilon\in(0,1)$, let $\widehat{L}_{2}$
be that value of $L_{2}$ for which $L_{4}=0$, i.e., for which $\theta_{6}=1.$
Using \eqref{theta6}, we can express $\widehat{L}_{2}$ as a function
of $K$ and $\varepsilon$: 
\begin{equation}
\widehat{L}_{2}(K,\varepsilon)=\ln{\left((1+\varepsilon)\frac{K-1-\varepsilon}{K-1}\right)}-\ln{\left(K+1-\frac{K}{\varepsilon}(K-1-\varepsilon)\ln{\frac{K-1}{K-1-\varepsilon}}\right)}.\label{L2kalap}
\end{equation}

\begin{prop}
\label{prop:hatL2} If $K\in(6.5,7)$ and $\varepsilon>0$ is small
enough, then $\widehat{L}_{2}(K,\varepsilon)\in(0,\varepsilon)$.\end{prop}
\begin{proof}
It is well-known that 
\[
\ln(1+x)=x-\frac{x^{2}}{2}+O\left(x^{3}\right)\mbox{ as }x\rightarrow0.
\]
In consequence, 
\begin{align*}
\ln{\frac{K-1}{K-1-\varepsilon}} & =\ln{\left(1+\frac{\varepsilon}{K-1-\varepsilon}\right)}\\
 & =\frac{\varepsilon}{K-1-\varepsilon}-\frac{\varepsilon^{2}}{2(K-1-\varepsilon)^{2}}+O\left(\varepsilon^{3}\right),
\end{align*}
and 
\begin{equation}
\frac{K}{\varepsilon}(K-1-\varepsilon)\ln{\frac{K-1}{K-1-\varepsilon}}=K-\frac{K\varepsilon}{2(K-1-\varepsilon)}+O\left(\varepsilon^{2}\right)\mbox{ as }\varepsilon\rightarrow0^{+}.\label{biz17}
\end{equation}
Applying the geometric series expansion $(1-x)^{-1}=\sum_{n=0}^{\infty}x^{n}$
with the choice $x=\varepsilon/(K-1)$, we easily deduce that 
\[
\frac{1}{K-1-\varepsilon}=\frac{1}{K-1}\cdot\frac{1}{1-\frac{\varepsilon}{K-1}}=\frac{1}{K-1}+O(\varepsilon),
\]
and thus 
\begin{equation}
\frac{K}{\varepsilon}(K-1-\varepsilon)\ln{\frac{K-1}{K-1-\varepsilon}}=K-\frac{K}{2(K-1)}\varepsilon+O\left(\varepsilon^{2}\right).\label{biz21}
\end{equation}
Using this, we get that 
\begin{equation}
\ln{\left(K+1-\frac{K}{\varepsilon}(K-1-\varepsilon)\ln{\frac{K-1}{K-1-\varepsilon}}\right)}=\frac{K}{2(K-1)}\varepsilon+O\left(\varepsilon^{2}\right).\label{biz19}
\end{equation}

Also note that 
\[
(1+\varepsilon)\frac{K-1-\varepsilon}{K-1}=1+\frac{K-2}{K-1}\varepsilon-\frac{1}{K-1}\varepsilon^{2},
\]
and thus 
\begin{equation}
\ln\left((1+\varepsilon)\frac{K-1-\varepsilon}{K-1}\right)=\frac{K-2}{K-1}\varepsilon+O\left(\varepsilon^{2}\right).\label{biz18}
\end{equation}
Subtracting (\ref{biz19}) from (\ref{biz18}), we conclude that 
\begin{equation}
\widehat{L}_{2}=\left(\frac{K-2}{K-1}-\frac{K}{2(K-1)}\right)\varepsilon+O\left(\varepsilon^{2}\right)=\frac{K-4}{2(K-1)}\varepsilon+O\left(\varepsilon^{2}\right).\label{biz22}
\end{equation}
Since $(K-4)/(2K-2)\in(0,1)$ for $K\in(6.5,7)$, we see that $\widehat{L}_{2}\in(0,\varepsilon)$
for all sufficiently small positive $\varepsilon$. 
\end{proof}
Consider the following subset of $U$:
\[
V=\left\{ (L_{2},K,\varepsilon)\colon\ \varepsilon\in(0,1),\ K\in\left(6.5,7\right)\mbox{ and }L_{2}\in\left(0,\widehat{L}_{2}(K,\varepsilon)\right)\right\} \subset U.
\]

\begin{rem}
\label{rem: theta6 L4} It is clear from \eqref{theta6} and \eqref{L4}
that $\theta_{6}$ and $L_{4}$ are strictly decreasing functions
of $L_{2}$. Hence, if $\varepsilon>0$ is small and $\left(L_{2},K,\varepsilon\right)\in V$,
then $\theta_{6}>1$ and $L_{4}>0$. \end{rem}

Now we are ready to clarify which fixed points of $L_{2}\mapsto F(L_{2},K,\varepsilon)$ yield periodic solutions.
\begin{prop}
\label{prop:H1H2stf} Assume that
\begin{itemize}
\item $(L_{2},K,\varepsilon)\in V$, $F(L_{2},K,\varepsilon)=L_{2}$ and
$\varepsilon>0$ is sufficiently small, 
\item $L_{1},L_{3},L_{4},L_{5}$ and $\theta_{i},$ $1\leq i\leq6$, are
defined as in (H6),
\item $y_{i}$, $1\leq i\leq10$, are defined as in (H7).
\end{itemize}
Then (H1)-(H5) hold.\end{prop}
\begin{proof}

The functions $y_{1},...,y_{10}$ can be defined by (\ref{eqsol:y1})-(\ref{eqsol:y10})
only if $L_{1},L_{2},L_{3}$, $L_{4}$ and $L_{5}$ are nonnegative. So
let us first show (H1). It is clear from (\ref{L3}) that $L_{3}>0$,
and we know from Remark \ref{rem: theta6 L4} that $L_{4}>0$. Corollary
\ref{cor: limit of K} and formula (\ref{L5}) yield that if $\epsilon$
(and hence $L_{2}$) tends to $0$, then $L_{5}$ tends to $\ln((K_{0}+1)/(K_{0}-1))$.
Similarly, Corollary \ref{cor: limit of K} and formula (\ref{L1})
give that 
\[
\lim_{\varepsilon\rightarrow0^{+}}L_{1}=\frac{1}{2}-\frac{1}{2}\ln\left(\frac{K_{0}+1}{K_{0}-1}\right)^{3}
\]
This limit is also positive, see \eqref{estimate}. Thus $L_{1}$
and $L_{5}$ are positive if $\varepsilon>0$ is small enough.

By Remark \ref{equivalency}, equation $F(L_{2},K,\varepsilon)=L_{2}$
and (\ref{L3})-(\ref{theta3}) imply (H2). Thus (H2) holds under
the assumptions of this proposition.

It is clear from (\ref{eqsol:y1})-(\ref{eqsol:y10}) that $y_{1},...,y_{10}$
are continuous. Recall that they are those solutions of the ordinary
differential equations (\ref{eq:y8})-(\ref{eq:y7}) that satisfy
the boundary conditions listed in (H4) for $t=0$. In addition, by
Remark \ref{equivalency}, equation $F(L_{2},K,\varepsilon)=L_{2}$
and (\ref{L3})-(\ref{theta3}) imply that the equations (\ref{sol:bk1})-(\ref{sol:bk10})
are true. This means that $y_{1},...,y_{10}$ satisfy the right-end
boundary conditions given in (H4). Summing up, (H4) is satisfied.

It remains to show (H3) and (H5).

As $K-1-\varepsilon$ is positive for $(L_{2},K,\varepsilon)\in V$,
we see from (\ref{eqsol:y1}) that $y_{1}$ strictly monotone increases
on $[0,L_{1}]$. As $y_{1}(0)=1+\varepsilon$, it follows that $y_{1}(t)>1+\varepsilon$
for $t\in(1,L_{1}]$. Thus $\theta_{1}=y_{1}(L_{1})>1+\varepsilon$.
Using Proposition \ref{prop:Fsorfejtes}.(iii), we obtain that $y_{2}(t)$,
$y_{3}(t)$ and $y_{4}(t)$ are greater than $1+\varepsilon$ for
all $t$ if $\varepsilon>0$ is small enough. It follows immediately
that $\theta_{2}=y_{2}\left(L_{2}\right)$, $\theta_{3}=y_{3}\left(L_{3}\right)$
and $\theta_{4}=y_{4}\left(L_{4}\right)$ are greater than $1+\varepsilon$
if $\varepsilon>0$ is small enough. It is clear from (\ref{eqsol:y5})
that $y_{5}$ is strictly monotone decreasing. As $y_{5}(L_{5})=1+\varepsilon$,
we deduce that $y_{5}(t)>1+\varepsilon$ for $t\in[0,L_{5})$. Summing
up the results of this paragraph, $\theta_{i}>1+\varepsilon$ for
$i\in\{1,2,3,4\}$, and $y_{i}(t)>1+\varepsilon$ for all $t$ in
the interior of the domain of $y_{i}$, where $i\in\{1,2,\ldots,5\}$.

By (\ref{eqsol:y6}) and (\ref{eqsol:y8}), $y_{6}$ and $y_{8}$
are strictly monotone decreasing on their whole domains. Differentiating
(\ref{eqsol:y7}) with respect to $t$ and using that $\theta_{5}>0$
by (\ref{theta5}), we conclude that 
\[
\dot{y}_{7}(t)=-\theta_{5}e^{-t}-(1+K)te^{-t}<0\quad\mbox{for }t\in[0,L_{3}],
\]
i.e., $y_{7}$ is also strictly monotone decreasing on $[0,L_{3}]$.
Hence 
\[
1+\varepsilon=y_{6}\left(0\right)>y_{6}(L_{2})=\theta_{5}=y_{7}(0)>y_{7}(L_{3})=\theta_{6}=y_{8}\left(0\right)>y_{8}\left(L_{4}\right)=1,
\]
and $y_{i}(s)\in(1,1+\varepsilon)$ for all $s$ in the interior of
the domain of $y_{i}$, where $i\in\{6,7,8\}$.

From (\ref{eqsol:y9}) and (\ref{eqsol:y10}) we conclude that $y_{9}$
and $y_{10}$ are strictly monotone decreasing on $\left[0,L_{2}+L_{5}\right]$
and on $[0,L_{3}]$, respectively. As we already know that $y_{9}$
and $y_{10}$ fulfill the boundary conditions given in (H4), it is
obvious that $y_{9}(t)\in(-1,1)$ for all $t\in(0,L_{2}+L_{5})$,
and $y_{10}(t)\in(-1-\varepsilon,-1)$ for all $t\in(0,L_{3})$.

We have verified (H3) and (H5), and hence the proof is complete.\end{proof}
\begin{cor}
\label{p is a sol}Under the assumptions of the previous proposition,
the $2\omega$-periodic function $p$, determined by (\ref{p:egyenletek})-(\ref{p:periodikus}),
satisfies the delay equation (\ref{Eq}) on $\mathbb{R}$. In addition,
the map 
\[
V\ni(L_{2},K,\varepsilon)\mapsto p_{0}\in C
\]
is continuous.\end{cor}
\begin{proof}
As $p(t)=-p(t-\omega)$ for all real $t$, and the nonlinearity $f_{K}$
is odd, it is enough to guarantee \eqref{eq:1} to prove that $p$ is a solution on $\mathbb{R}$. By Proposition \ref{prop:H1H2stf},
the properties listed in (H1)-(H5) are true. We have already pointed
out in Section 2 that -- under hypotheses (H1)-(H5)-- equation (\ref{eq:1})
is equivalent to the ordinary equations
\begin{itemize}
\item \eqref{eq:y8}-\eqref{eq:y1} on $[0,\tau_{1}]$,
\item \eqref{eq:y2}-\eqref{eq:y4} on $[\tau_{1},\tau_{2}]$,
\item \eqref{eq:y5}-\eqref{eq:y6} on $[\tau_{2},\tau_{3}]$,
\item and \eqref{eq:y7} on $[\tau_{3},\omega]$.
\end{itemize}
By assumption (H7), (\ref{eq:y8})-(\ref{eq:y7}) hold. So \eqref{eq:1}
is satisfied too.

Recall the observation that under hypothesis (H6), $L_{i}$, $i\in\{1,3,4,5\}$,
and $\theta_{i},$ $1\leq i\leq6$, are continuous functions of $(L_{2},K,\varepsilon)$
on $V$. From this, from the definitions (\ref{eqsol:y1})-(\ref{eqsol:y10})
of $y_{1},\ldots,y_{10}$ and from (\ref{p:egyenletek})-(\ref{p:periodikus})
one can deduce in a straightforward way that the initial function $p_0$ varies continuously with $(L_{2},K,\varepsilon)$. We leave the details to the reader.
\end{proof}

\section{The saddle-node bifurcation of $F_{\varepsilon}$}

For $\varepsilon\in(0,1)$, let 
\[
U_{\varepsilon}=\left(-\varepsilon,\varepsilon\right)\times(6.5,7)
\]
and define 
\[
F_{\varepsilon}:U_{\varepsilon}\ni(L_{2},K)\mapsto F\left(L_{2},K,\varepsilon\right)\in\mathbb{R}.
\]
Appendix A of this paper calculates certain partial derivatives
of $F_{\varepsilon}$ and shows that $\partial F_{\varepsilon}/\partial K$
and $\partial^{2}F_{\varepsilon}/\partial L_{2}^{2}$ are both continuous
on $U_{\varepsilon}$. One can actually show that $F_{\varepsilon}\in C^{2}(U_{\varepsilon})$.
We omit the complete proof of this claim.

In this section we consider $\varepsilon$ to be a fixed and sufficiently
small positive number and show that $F_{\varepsilon}$ undergoes a
saddle-node bifurcation as $K$ increases. 

The first two propositions show that if $L_{2}\in(-\varepsilon,\varepsilon)$,
then the equation $F_{\varepsilon}(L_{2},K)=L_{2}$ can be solved
for $K.$
\begin{prop}
\label{prop:iftseged} If $\varepsilon>0$ is sufficiently small,
then there exists $K_{\varepsilon}\in(6.5,7)$ so that $F_{\varepsilon}\left(0,K_{\varepsilon}\right)=0$. \end{prop}
\begin{proof}
Using Proposition \ref{prop:Fsorfejtes}, we obtain that for any $K\in\left(6.5,7\right)$,
\begin{equation}
\lim_{\varepsilon\rightarrow0^{+}}F_{\varepsilon}(0,K)=K-\sqrt{\frac{(K+1)^{3}}{e(K-1)}}-\frac{K+1}{K-1}.\label{limit of F}
\end{equation}
One can show by elementary calculations that the sign of \eqref{limit of F}
is the same as the sign of 
\[
w\left(K\right)=e-\frac{(K+1)^{3}(K-1)}{(K^{2}-2K-1)^{2}},
\]
which expression is slightly easier to handle. As 
\[
w(6.5)=e-\frac{37125}{12769}<e-\frac{36400}{13000}=e-\frac{28}{10}<0
\]
and 
\[
w(7)=e-\frac{768}{289}>e-\frac{768}{288}=e-\frac{8}{3}>0,
\]
one sees that $\lim_{\varepsilon\rightarrow0^{+}}{F_{\varepsilon}(0,6.5)}<0$
and $\lim_{\varepsilon\rightarrow0^{+}}{F_{\varepsilon}(0,7)}>0.$
Hence, if $\varepsilon>0$ is small enough, then $F_{\varepsilon}(0,6.5)<0$
and $F_{\varepsilon}(0,7)>0$. As $(6.5,7)\ni K\mapsto F_{\varepsilon}(0,K)\in\mathbb{R}$
is continuous, the existence of $K_{\varepsilon}$ follows from the
intermediate value theorem. \end{proof}
\begin{prop}
\label{prop:ift} For all sufficiently small positive $\varepsilon$
and $L_{2}\in(-\varepsilon,\varepsilon)$, the equation $F_{\varepsilon}(L_{2},K)=L_{2}$
has a unique solution $K$ in $(6.5,7)$. Furthermore, this solution
can be given as $K=\varphi_{\varepsilon}(L_{2})$, where $\varphi_{\varepsilon}:(-\varepsilon,\varepsilon)\rightarrow(6.5,7)$
is continuous and $\varphi_{\varepsilon}(0)=K_{\varepsilon}$. \end{prop}
\begin{proof}
Let $J_{\varepsilon}=\left(6.5-K_{\varepsilon},7-K_{\varepsilon}\right)$
and introduce the map 
\[
G_{\varepsilon}:(-\varepsilon,\varepsilon)\times J_{\varepsilon}\ni(L_{2},K)\mapsto F_{\varepsilon}(L_{2},K+K_{\varepsilon})-L_{2}\in\mathbb{R}.
\]
Then $G_{\varepsilon}(0,0)=0$ and $G$ is continuously differentiable
(see Propositions \ref{prop:Kderivalt} and \ref{prop:L2 szerinti derivalt}
in  Appendix A). We look for the solution $K$ of $G_{\varepsilon}(L_{2},K)=0$
for any small $\varepsilon>0$ and for any $L_{2}\in(-\varepsilon,\varepsilon)$. 

Let 
\[
A:=\frac{\partial G_{\varepsilon}(0,0)}{\partial K}=\frac{\partial F_{\varepsilon}(0,K_{\varepsilon})}{\partial K}.
\]
By Corollary {\ref{cor:Kszerintiderivlathatarertek}},
we may assume that $A$ is nonzero. 

Finding a solution $K$ to $G_{\varepsilon}(L_{2},K)=0$ is equivalent
to finding a fixed point of $T_{\varepsilon,L_{2}}$, where 
\[
T_{\varepsilon,L_{2}}:J_{\varepsilon}\ni K\mapsto K-A^{-1}G_{\varepsilon}(L_{2},K)\in\mathbb{R}.
\]

Choose a constant $q\in\left(0,1\right)$ independent from $K,\varepsilon$
and $L_{2}$. We claim that $T_{\varepsilon,L_{2}}$ is a uniform
contraction on an appropriate subset of $J_{\varepsilon}$: if $\eta>0$
is small enough, then $[-\eta,\eta]\subseteq J_{\varepsilon}$, 
\begin{equation}
|T_{\varepsilon,L_{2}}(K)|\leq\eta\mbox{ for }K\in[-\eta,\eta],\label{biz25}
\end{equation}
and 
\begin{equation}
|T_{\varepsilon,L_{2}}(K_{1})-T_{\varepsilon,L_{2}}(K_{2})|<q|K_{1}-K_{2}|\mbox{ for }K_{1},K_{2}\in[-\eta,\eta].\label{biz26}
\end{equation}
 Set $\eta>0$ so small that $[-\eta,\eta]\subseteq J_{\varepsilon}$.
Using Lagrange's mean value theorem, we get that for $K_{1},K_{2}\in[-\eta,\eta],$
\begin{align}
\left|T_{\varepsilon,L_{2}}(K_{1})-T_{\varepsilon,L_{2}}(K_{2})\right| & \leqslant\sup_{\left|\bar{K}\right|<\eta}\left|1-A^{-1}\frac{\partial G_{\varepsilon}(L_{2},\bar{K})}{\partial K}\right|\left|K_{1}-K_{2}\right|.\label{biz27}
\end{align}
We see from Proposition \ref{prop:Kderivalt} that 
\begin{align}
\frac{\partial G_{\varepsilon}(L_{2},\bar{K})}{\partial K} & =1-e^{-\frac{1}{2}}\left(\frac{3}{2}\sqrt{\frac{K_{\varepsilon}+\bar{K}+1}{K_{\varepsilon}+\bar{K}-1}}-\frac{1}{2}\sqrt{\left(\frac{K_{\varepsilon}+\bar{K}+1}{K_{\varepsilon}+\bar{K}-1}\right)^{3}}\right)\label{biz28}\\
 & \quad+\frac{2}{(K_{\varepsilon}+\bar{K}-1)^{2}}+O(\varepsilon).\notag
\end{align}
Therefore there exist $\varepsilon_{0}>0$ and $\eta_{0}>0$ such
that if $\varepsilon\in(0,\varepsilon_{0})$, $L_{2}\in(-\varepsilon,\varepsilon)$
and $\bar{K}\in(-\eta_{0},\eta_{0})$, then 
\[
\left|1-A^{-1}\frac{\partial G_{\varepsilon}(L_{2},\bar{K})}{\partial K}\right|<q,
\]
that is, (\ref{biz26}) is satisfied for any $\varepsilon\in(0,\varepsilon_{0})$,
$L_{2}\in(-\varepsilon,\varepsilon)$ and $\eta\in\left(0,\eta_{0}\right)$.
Next, using the Taylor expansion of $G_{\varepsilon}$, we obtain
that 
\[
T_{\varepsilon,L_{2}}(K)=K-A^{-1}G_{\varepsilon}(L_{2},K)=-A^{-1}\left(\frac{\partial G_{\varepsilon}(0,0)}{\partial L_{2}}L_{2}+O\left(L_{2}^{2}+K^{2}\right)\right)
\]
as $L_{2}\rightarrow0$ and $K\rightarrow0$. In consequence, if $|K|\leq\eta<\eta_{0}$
and $|L_{2}|<\varepsilon<\varepsilon_{0}$, then 
\[
\left|T_{\varepsilon,L_{2}}(K)\right|<|A|^{-1}\left|\frac{\partial G_{\varepsilon}(0,0)}{\partial L_{2}}\right|\varepsilon+C\left(\varepsilon^{2}+\eta^{2}\right)
\]
with some constant $C>0$. Fix $\eta_{1}<\eta_{0}$ so small that
$C\eta_{1}^{2}<\eta_{1}/2$. Now set $\varepsilon_{1}<\varepsilon_{0}$
such that 
\[
|A|^{-1}\left|\frac{\partial G_{\varepsilon}(0,0)}{\partial L_{2}}\right|\varepsilon+C\varepsilon^{2}<\frac{\eta_{1}}{2}
\]
for all $\varepsilon\in(0,\varepsilon_{1})$. Then (\ref{biz25})
holds for $\eta=\eta_{1}$, $\varepsilon\in(0,\varepsilon_{1})$ and
$L_{2}\in(-\varepsilon,\varepsilon)$.

Summing up, we conclude that $T_{\varepsilon,L_{2}}$ is a uniform
contraction from $[-\eta_{1},\eta_{1}]$ to $[-\eta_{1},\eta_{1}]$
for all $\varepsilon\in(0,\varepsilon_{1})$ and $L_{2}\in(-\varepsilon,\varepsilon)$.
By the Banach fixed point theorem, $T_{\varepsilon,L_{2}}$ has a
unique fixed point $\psi_{\varepsilon}(L_{2})$ in $[-\eta_{1},\eta_{1}]$,
see Theorem $3.1$ of Chapter $0$ in \cite{Hale}. Since $L_{2}\mapsto T_{\varepsilon,L_{2}}(K)$
is continuous for each $K$, it follows from Theorem 3.2 of Chapter
$0$ in \cite{Hale} that $\psi_{\varepsilon}$ is continuous in $L_{2}$.
It is clear that $\psi_{\varepsilon}(0)=0$. Set $\varphi_{\varepsilon}(K):=K_{\varepsilon}+\psi_{\varepsilon}(K).$ 
\end{proof}
Now we are ready to verify the saddle-node bifurcation of the fixed
points of $F_{\varepsilon}$.
\begin{prop}
\label{prop:bifurcation} For all sufficiently small positive $\varepsilon$,
one can give $K^{*}=K^{*}(\varepsilon)\in\left(6.5,7\right)$ and
$L_{2}^{*}=L_{2}^{*}(\varepsilon)\in(0,\widehat{L}_{2}(K,\varepsilon))$
such that $F_{\varepsilon}$ undergoes a saddle-node bifurcation at
$(L_{2}^{*},K^{*})$: there exist a neighborhood $\mathcal{U}$ of
$L_{2}^{*}$ in $(0,\widehat{L}_{2}(K^{*},\varepsilon))$ and a constant
$\delta_{1}>0$ such that 
\begin{itemize}
\item the map $F_{\varepsilon}(\cdot,K)$ has no fixed point in $\mathcal{U}$
for $K\in(K^{*}-\delta_{1},K^{*})$, 
\item $L_{2}^{*}$ is the unique fixed point of $F_{\varepsilon}\left(\cdot,K^{*}\right)$
in $\mathcal{U}$, 
\item $F_{\varepsilon}(\cdot,K)$ has exactly two fixed points in $\mathcal{U}$
for $K\in(K^{*},K^{*}+\delta_{1})$, and both fixed points converge to $L_{2}^{*}$ as $K\rightarrow K^{*}$. 
\end{itemize}
\end{prop}
\begin{proof}
By Section 21.1A in \cite{Wiggins}, $F_{\varepsilon}$ undergoes
a saddle-node bifurcation at $(L_{2}^{*},K^{*})$ if 
\begin{align*}
F_{\varepsilon}(L_{2}^{*},K^{*}) & =L_{2}^{*},\\
\frac{\partial}{\partial L_{2}}F_{\varepsilon}(L_{2}^{*},K^{*}) & =1,\\
\frac{\partial}{\partial K}F_{\varepsilon}(L_{2}^{*},K^{*}) & \neq0,\\
\frac{\partial^{2}}{\partial L_{2}^{2}}F_{\varepsilon}(L_{2}^{*},K^{*}) & \neq0.
\end{align*}
Furthermore, if 
\[
\frac{\frac{\partial^{2}}{\partial L_{2}^{2}}F_{\varepsilon}(L_{2}^{*},K^{*})}{\frac{\partial}{\partial K}F_{\varepsilon}(L_{2}^{*},K^{*})}<0,
\]
then the fixed points of $F_{\varepsilon}$ appear for $K\geq K^{*}$. 

For all small enough positive $\varepsilon$, Proposition \ref{prop:ift}
gives a continuous map $\varphi_{\varepsilon}:(-\varepsilon,\varepsilon)\rightarrow(6.5,7)$
such that 
\[
F_{\varepsilon}\left(L_{2},\varphi_{\varepsilon}(L_{2})\right)=L_{2}\mbox{ for all }L_{2}\in(-\varepsilon,\varepsilon).
\]
It is clear from Corollary \ref{cor:L2derivaltjta} that if $\varepsilon>0$
is sufficiently small, then 
\[
\frac{\partial}{\partial L_{2}}F_{\varepsilon}(0,\varphi_{\varepsilon}(0))>1\quad\mbox{and}\quad\frac{\partial}{\partial L_{2}}F_{\varepsilon}\left(\widehat{L}_{2},\varphi_{\varepsilon}\left(\widehat{L}_{2}\right)\right)<1.
\]
As $F_{\varepsilon}$ is continuously differentiable with respect
to $L_{2}$, and $\varphi_{\epsilon}$ is continuous, it is clear
that 
\[
(-\varepsilon,\varepsilon)\ni L_{2}\mapsto\frac{\partial}{\partial L_{2}}F_{\varepsilon}(L_{2},\varphi_{\varepsilon}(L_{2}))\in\mathbb{R}
\]
is also continuous. It follows from the intermediate value theorem
that there exists $L_{2}^{*}\in(0,\widehat{L}_{2})$ such that 
\[
\frac{\partial}{\partial L_{2}}F_{\varepsilon}\left(L_{2}^{*},\varphi_{\varepsilon}\left(L_{2}^{*}\right)\right)=1.
\]
Let $K^{*}:=\varphi_{\varepsilon}\left(L_{2}^{*}\right)\in(6.5,7).$

We see from Corollary \ref{cor:Kszerintiderivlathatarertek} and from
Proposition \ref{prop:L2L2 szerinti derivalt} that we may assume
that 
\[
\frac{\partial}{\partial K}F_{\varepsilon}\left(L_{2}^{*},K^{*}\right)>0\quad\mbox{and}\quad\frac{\partial^{2}}{\partial L_{2}^{2}}F_{\varepsilon}\left(L_{2}^{*},K^{*}\right)<0.
\]

Hence $F_{\varepsilon}$ undergoes a saddle-node bifurcation at $(L_{2}^{*},K^{*})$,
and the fixed points appear for $K\geq K^{*}$.
\end{proof}

\section{The delay equation has no other types of periodic solutions locally}

In this section choose $\varepsilon>0$ so small that Proposition
\ref{prop:bifurcation} holds, i.e., $F_{\varepsilon}$ undergoes
a saddle-node bifurcation at $(L_{2}^{*}(\varepsilon),K^{*}(\varepsilon))$,
where $(L_{2}^{*}(\varepsilon),K^{*}(\varepsilon),\varepsilon)\in V$. 

From now on, let $p\colon\mathbb{R}\rightarrow\mathbb{R}$ denote that periodic solution that is given by Corollary
\ref{p is a sol} specially for $(L_{2}^{*}(\varepsilon),K^{*}(\varepsilon),\varepsilon)$.
Then $p$ is the concatenation of certain auxiliary functions $y_{1},\ldots,y_{10}$
as in \eqref{p:egyenletek}-\eqref{p:periodikus}, and its minimal
period is $2\omega$. The functions $y_{1},\ldots,y_{10}$ satisfy
(H1)-(H5) with some parameters $L_{i}>0$, $i\in\{1,2,...,5\}$, and
$\theta_{i}$, $i\in\left\{ 1,\ldots,6\right\} $.

In order to complete the proof of the main theorem, it remains to
verify that all periodic solutions of the delay equation \eqref{Eq}
derive from fixed points of $F$ - at least locally, in an open ball
centered at $p_{0}$. 

First let us recall the results of Propositions 5.1 and 5.2 in \cite{Krisztin-Vas}.
\begin{prop}
\label{prop:monotonicity property} Suppose that $\bar{p}:\mathbb{R}\rightarrow\mathbb{R}$
is an arbitrary periodic solution of \eqref{Eq} with minimal period
$2\bar{\omega}$.\\
(i) If $t_{0}\in\mathbb{R}$ and $t_{1}\in\left(t_{0},t_{0}+2\bar{\omega}\right)$
are chosen so that $\bar{p}\left(t_{0}\right)=\min_{t\in\mathbb{R}}\bar{p}\left(t\right)$
and $\bar{p}\left(t_{1}\right)=\max_{t\in\mathbb{R}}\bar{p}\left(t\right)$,
then $\bar{p}$ is monotone nondecreasing on $\left(t_{0},t_{1}\right)$
and monotone nonincreasing on $\left(t_{1},t_{0}+2\bar{\omega}\right)$.\\
(ii) If $0\in\bar{p}(\mathbb{R})$, then $\bar{p}(t)=-\bar{p}(t-\bar{\omega})$
for all real $t$.
\end{prop}
The main result of this section is the following.
\begin{prop}
\label{prop:no other per sol}Let $\bar{p}:\mathbb{R}\rightarrow\mathbb{R}$
be a periodic solution of \eqref{Eq} for some parameter $\bar{K}$
with minimal period $2\bar{\omega}$. If $\left|\bar{K}-K^{*}\left(\varepsilon\right)\right|$
and $\left\Vert \bar{p}_{0}-p_{0}\right\Vert $ are small enough and
$\bar{p}\left(-1\right)=1+\varepsilon$, then one can give parameters
$\bar{L}_{i}>0$, $i\in\{1,2,...,5\}$, $\bar{\theta}_{i}$, $i\in\left\{ 1,\ldots,6\right\} $,
and continuous functions $\bar{y}_{i}$, $i\in\left\{ 1,\ldots,10\right\} $,
such that (H1)-(H5) hold, and $\bar{p}$ is the concatenation of $\bar{y}_{1},\ldots,\bar{y}_{10}$
as follows:
\begin{align}
\bar{p}(t-1) & =\bar{y}_{1}(t)\quad\mbox{for }t\in\left[0,\bar{L}_{1}\right],\notag\label{bar p:egyenletek}\\
\bar{p}\left(t-1+\bar{L}_{1}\right) & =\bar{y}_{2}(t)\quad\mbox{for }t\in\left[0,\bar{L}_{2}\right],\notag\\
\bar{p}\left(t-1+\bar{L}_{1}+\bar{L}_{2}\right) & =\bar{y}_{3}(t)\quad\mbox{for }t\in\left[0,\bar{L}_{3}\right],\notag\\
\bar{p}\left(t-1+\bar{L}_{1}+\bar{L}_{2}+\bar{L}_{3}\right) & =\bar{y}_{4}(t)\quad\mbox{for }t\in\left[0,\bar{L}_{4}\right],\notag\\
\bar{p}\left(t-1+\bar{L}_{1}+\bar{L}_{2}+\bar{L}_{3}+\bar{L}_{4}\right) & =\bar{y}_{5}(t)\quad\mbox{for }t\in\left[0,\bar{L}_{5}\right],\\
\bar{p}\left(t-1+\bar{\tau}_{1}\right) & =\bar{y}_{6}(t)\quad\mbox{for }t\in\left[0,\bar{L}_{2}\right],\notag\\
\bar{p}\left(t-1+\bar{\tau}_{1}+\bar{L}_{2}\right) & =\bar{y}_{7}(t)\quad\mbox{for }t\in\left[0,\bar{L}_{3}\right],\notag\\
\bar{p}\left(t-1+\bar{\tau}_{1}+\bar{L}_{2}+\bar{L}_{3}\right) & =\bar{y}_{8}(t)\quad\mbox{for }t\in\left[0,\bar{L}_{4}\right],\notag\\
\bar{p}\left(t-1+\bar{\tau}_{2}\right) & =\bar{y}_{9}(t)\quad\mbox{for }t\in\left[0,\bar{L}_{2}+\bar{L}_{5}\right],\notag\\
\bar{p}\left(t-1+\bar{\tau}_{3}\right) & =\bar{y}_{10}(t)\quad\mbox{for }t\in\left[0,\bar{L}_{3}\right],\notag
\end{align}
where 
\begin{equation}
\bar{\tau}_{1}=\sum_{i=1}^{5}{\bar{L}_{i}},\ \bar{\tau}_{2}=\bar{\tau}_{1}+\bar{L}_{2}+\bar{L}_{3}+\bar{L}_{4},\ \bar{\tau}_{3}=\bar{\tau}_{2}+\bar{L}_{2}+\bar{L}_{5}\ \mbox{and}\ \bar{\omega}=\bar{\tau}_{3}+\bar{L}_{3}.\label{bar tau}
\end{equation}
 In addition, 
\begin{equation}
\bar{p}(t)=-\bar{p}(t-\bar{\omega})\quad\mbox{for all }t\in\mathbb{R}.\label{symmetry property}
\end{equation}
In consequence, $\bar{L}_{2}$ is the fixed point of $\left(0,\varepsilon\right)\ni L_{2}\mapsto F\left(L_{2},K,\varepsilon\right)\in\mathbb{R}$. \end{prop}
\begin{proof}
As $\left\Vert \bar{p}_{0}-p_{0}\right\Vert $ is small, we may assume that $0\in\bar{p}(\mathbb{R})$. So the symmetry property \eqref{symmetry property} holds by Proposition \ref{prop:monotonicity property}.(ii), and it suffices to investigate $\bar{p}$ on $[-1,-1+\bar{\omega}]$.
We prove the theorem by explicitly determining the auxiliary functions
$\bar{y}_{1},\ldots,\bar{y}_{10}$, the parameters $\bar{L}_{1},\ldots\bar{L}_{5}$
as the lengths of their domains, and the parameters $\bar{\theta}_{1},\ldots,\bar{\theta}_{6}$
as their boundary values.

1. One can easily prove that $\left|\bar{\omega}-\omega\right|$ is
arbitrary small provided $\left|\bar{K}-K^{*}\left(\varepsilon\right)\right|$
and $\left\Vert \bar{p}_{0}-p_{0}\right\Vert $ are small enough.
Hence, by the smallness of $\left\Vert \bar{p}_{0}-p_{0}\right\Vert $
and by the continuity of the solution operator in forward time, one
can achieve that $\left\Vert \bar{p}_{-1+2\bar{\omega}}-p_{-1+2\omega}\right\Vert $
is arbitrary small too. By periodicity, this means that $\left\Vert \bar{p}_{-1}-p_{-1}\right\Vert $
is arbitrary small. As we shall see, this property is of key role.

It is also straightforward to obtain the subsequent properties of
$p_{-1}$ (shown Fig.~\ref{fig:ketgrafikon}) by using (H2), the
equations \eqref{p:egyenletek} and the fact that $p(t)=-p(t-\omega)$
for all $t\in\mathbb{R}$: 
\begin{equation}
p(t)>1+\varepsilon\mbox{ for }t\in[-2,-2+L_{1})\quad\mbox{and}\quad p\left(-2+L_{1}\right)=1+\varepsilon,\label{property1}
\end{equation}
\begin{multline}
p(t)\in\left(-1,1\right)\ \mbox{for }t\in\left(-2+\tau_{1}-L_{5},-2+\tau_{1}+L_{2}\right),\\
p(-2+\tau_{1}-L_{5})=1\ \mbox{and\ }p\left(-2+\tau_{1}+L_{2}\right)=-1,\label{property5}
\end{multline}
\begin{multline}
p(t)<-1-\varepsilon\ \mbox{for }t\in(-2+\tau_{2}-L_{4},-2+\omega+L_{1})\ \mbox{and}\\
p(-2+\tau_{2}-L_{4})=-1-\varepsilon.\label{property5-1}
\end{multline}

\begin{figure}[b]
\centering{}\centering \includegraphics[angle=90,width=0.75\textwidth]{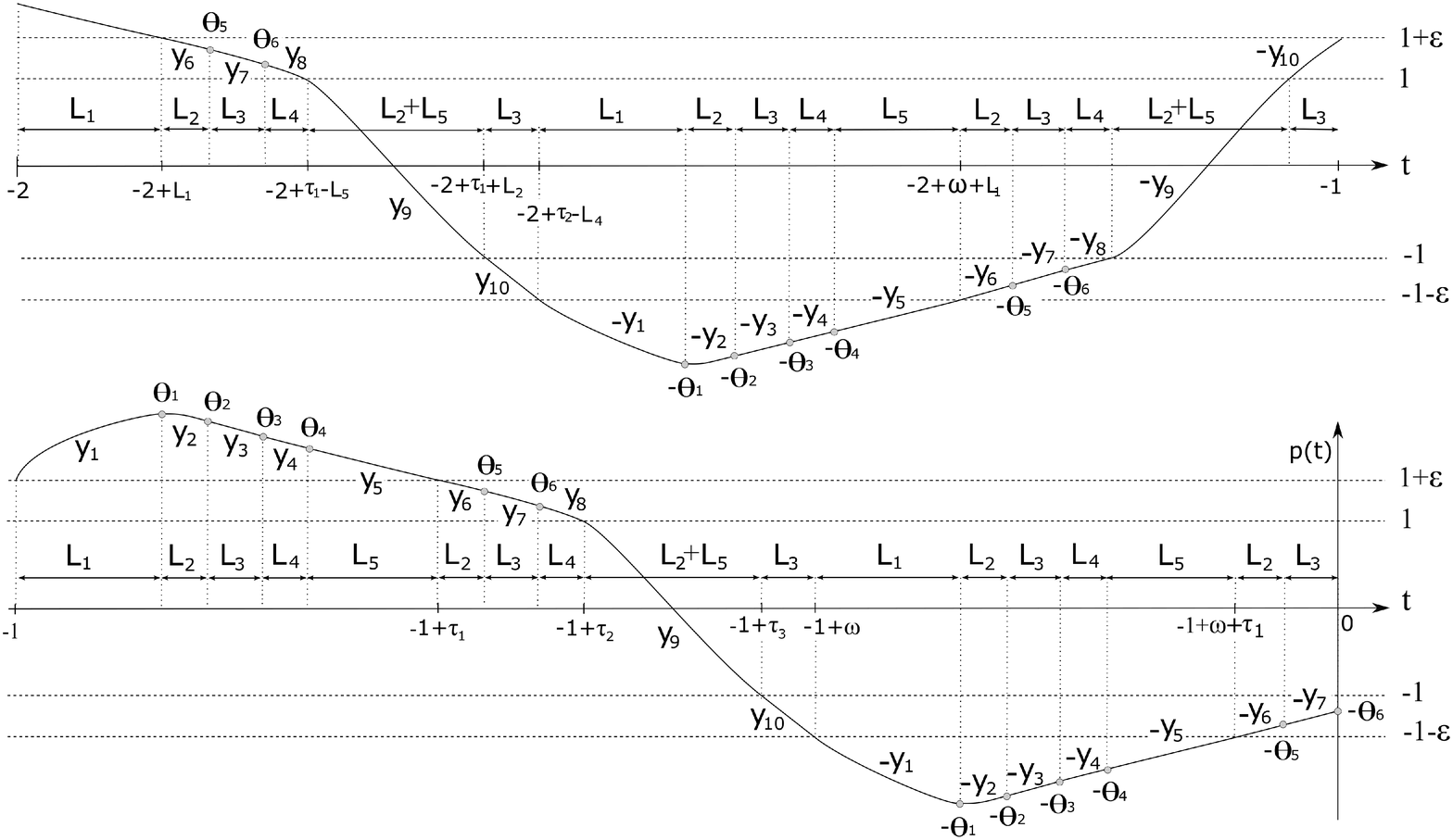}
\caption{The plot of p on {[}-2,-1{]} and on {[}-1,0{]}.}
\label{fig:ketgrafikon} 
\end{figure}

2. \emph{The parameters $\bar{L}_{1},$ $\bar{\theta}_{1}$ and the
function $\bar{y}_{1}$.} Since $\left\Vert \bar{p}_{-1}-p_{-1}\right\Vert $
is arbitrarily small and \eqref{property1} holds, one can give $\bar{L}_{1}>0$
arbitrarily close to $L_{1}$ such that 
\begin{equation}
\bar{p}(t)>1+\varepsilon\mbox{ for }t\in[-2,-2+\bar{L}_{1})\quad\mbox{and}\quad\bar{p}\left(-2+\bar{L}_{1}\right)=1+\varepsilon.\label{barp0}
\end{equation}
Define $\bar{y}_{1}\in C\left([0,\bar{L}_{1}],\mathbb{R}\right)$
by $\bar{y}_{1}\left(t\right)=\bar{p}\left(t-1\right)$ for all $t\in[0,\bar{L}_{1}]$.
As $\bar{p}\left(-1\right)=1+\varepsilon$, it is clear that $\bar{y}_{1}\left(0\right)=1+\varepsilon.$
Set $\bar{\theta}_{1}=\bar{y}_{1}(\bar{L}_{1})=\bar{p}(-1+\bar{L}_{1})$.
Then $\bar{y}_{1}$ obviously fulfills the conditions in (H4). Also
note that 
\[
\dot{\bar{y}}_{1}\left(t\right)=\dot{\bar{p}}\left(t-1\right)=-\bar{p}\left(t-1\right)+f_{K}\left(\bar{p}(t-2)\right)=-\bar{y}_{1}\left(t\right)+K\quad\mbox{for }t\in\left[0,\bar{L}_{1}\right].
\]
Considering the solution of this linear equation, it is clear that
$\bar{y}_{1}$ and $\bar{\theta}_{1}$ satisfy the conditions in (H3)
and (H5): $\bar{\theta}_{1}>1+\varepsilon$ and $\bar{y}_{1}(t)>1+\varepsilon$
for all $t\in(0,\bar{L}_{1})$.

3. At this point we do not have enough information to determine $\bar{y}_{2}$,
$\bar{y}_{3}$ or $\bar{y}_{4}$. Next we define \emph{$\bar{y}_{5}$,
$\bar{y}_{6}$ and those parameters that are related to them (namely,
$\bar{L}_{2},\bar{L}_{5},\bar{\theta}_{4}$ and $\bar{\theta}_{5}$).}
Recall that $p\left(-1+\tau_{1}\right)=1+\varepsilon$. Under the
assumptions of the proposition, one can give a minimal $\bar{\tau}_{1}>0$
(arbitrarily close to $\tau_{1}$) such that $\bar{p}(-1+\bar{\tau}_{1})=1+\varepsilon$.
By \eqref{property5} and by the convergence of $\bar{p}_{-1}$ to
$p_{-1}$, we may assume the existence of $\bar{L}_{5}>0$ and $\bar{L}_{2}>0$
(arbitrarily close to $L_{5}$ and $L_{2}$, respectively) such that
\begin{equation}
\bar{p}(t)\in\left(-1,1\right)\quad\mbox{for }t\in\left(-2+\bar{\tau}_{1}-\bar{L}_{5},-2+\bar{\tau}_{1}+\bar{L}_{2}\right).\label{barp1}
\end{equation}
Choose $\bar{L}_{2}$ and $\bar{L}_{5}$ such that the time interval
in \eqref{barp1} is maximal: 
\begin{equation}
\bar{p}\left(-2+\bar{\tau}_{1}-\bar{L}_{5}\right)=1\quad\mbox{and}\quad\bar{p}\left(-2+\bar{\tau}_{1}+\bar{L}_{2}\right)=-1.\label{barp2}
\end{equation}
This means that 
\[
\dot{\bar{p}}(t)=-\bar{p}(t)\mbox{ for }t\in\left(-1+\bar{\tau}_{1}-\bar{L}_{5},-1+\bar{\tau}_{1}+\bar{L}_{2}\right).
\]
As $p$ is positive on $(-1+\tau_{1}-L_{5},-1+\tau_{1}+L_{2})$, we
may assume that $\bar{p}$ is positive on $(-1+\bar{\tau}_{1}-\bar{L}_{5},-1+\bar{\tau}_{1}+\bar{L}_{2})$,
and hence we deduce from the last ordinary differential equation that
$\bar{p}$ strictly decreases on this interval. Set $\bar{y}_{5}\in C\left([0,\bar{L}_{5}],\mathbb{R}\right)$
and $\bar{y}_{6}\in C\left([0,\bar{L}_{2}],\mathbb{R}\right)$ by
\[
\bar{y}_{5}\left(t\right)=\bar{p}\left(t-1+\bar{\tau}_{1}-\bar{L}_{5}\right)\quad\mbox{for all }t\in\left[0,\bar{L}_{5}\right],
\]
and 
\[
\bar{y}_{6}\left(t\right)=\bar{p}\left(t-1+\bar{\tau}_{1}\right)\quad\mbox{for all }t\in\left[0,\bar{L}_{2}\right].
\]
Then $\bar{y}_{5}$ and $\bar{y}_{6}$ are strictly decreasing functions
with $\bar{y}_{5}(\bar{L}_{5})=\bar{y}_{6}\left(0\right)=\bar{p}(-1+\bar{\tau}_{1})=1+\varepsilon$.
Let 
\begin{equation}
\bar{\theta}_{4}=\bar{y}_{5}(0)=\bar{p}\left(-1+\bar{\tau}_{1}-\bar{L}_{5}\right)\quad\mbox{and}\quad\bar{\theta}_{5}=\bar{y}_{6}\left(\bar{L}_{2}\right)=\bar{p}\left(-1+\bar{\tau}_{1}+\bar{L}_{2}\right),\label{barp3}
\end{equation}
so that $\bar{y}_{5}$ and $\bar{y}_{6}$ satisfy the conditions in
(H4). With these choices, $\bar{\theta}_{4}$ and $\bar{\theta}_{5}$
are arbitrarily close to $\theta_{4}=p(-1+\tau_{1}-L_{5})$ and $\theta_{5}=p(-1+\tau_{1}+L_{2})$,
respectively, and hence one can achieve that $\bar{\theta}_{4}>1+\varepsilon$
and $\bar{\theta}_{5}\in(1,1+\varepsilon)$ -- as required by (H3).
The monotonicity of $\bar{y}_{5}$ and $\bar{y}_{6}$ guarantee that
$\bar{y}_{5}$ and $\bar{y}_{6}$ also fulfill the next conditions
in (H5): $\bar{y}_{5}\left(t\right)>1+\varepsilon$ for all $t\in(0,\bar{L}_{5})$
and $\bar{y}_{6}(t)\in(1,1+\varepsilon)$ for all $t\in(0,\bar{L}_{2})$.

4. \emph{The functions $\bar{y}_{8}$, $\bar{y}_{9},$ $\bar{y}_{10}$
and the parameters $\bar{L}_{3},\bar{L}_{4},\bar{\theta}_{6}$.} Let
$\bar{\tau}_{2}\in(-1+\bar{\tau}_{1},-1+\bar{\omega})$ be minimal
with $\bar{p}(-1+\bar{\tau}_{2})=1$. Such $\bar{\tau}_{2}$ exists
because $\bar{p}_{0}$, $\bar{\tau}_{1}$, $\bar{\omega}$ is arbitrarily
close to $p_{0}$, $\tau_{1}$, $\omega$, respectively. Then $\bar{\tau}_{2}$
converges to $\tau_{2}$ as $\bar{p}_{0}$ converges to $p_{0}$. 

The convergence of $\bar{p}_{-1}$ to $p_{-1}$ and property \eqref{property5-1}
ensure the existence of $\bar{L}_{4}>0$ (arbitrarily close to $L_{4}$)
so that 
\begin{equation}
\bar{p}\left(-2+\bar{\tau}_{2}-\bar{L}_{4}\right)=-1-\varepsilon.\label{barp5}
\end{equation}
and Let $\bar{y}_{8}$ be the continuous function on $[0,\bar{L}_{4}]$
given by 
\[
\bar{y}_{8}(t)=\bar{p}\left(t-1+\bar{\tau}_{2}-\bar{L}_{4}\right)\quad\mbox{for all }t\in\left[0,\bar{L}_{4}\right],
\]
and let 
\begin{equation}
\bar{\theta}_{6}=\bar{y}_{8}\left(0\right)=\bar{p}\left(-1+\bar{\tau}_{2}-\bar{L}_{4}\right).\label{barp6}
\end{equation}
Then $\bar{\theta}_{6}$ is arbitrarily close to $\theta_{6}=p(-1+\tau_{2}-L_{4})$,
and therefore we may assume that $\bar{\theta}_{6}\in(1,1+\varepsilon)$
(see again (H3)). At the right-end point of its domain, $\bar{y}_{8}$
takes the value $\bar{y}_{8}\text{(}\bar{L}_{4})=\bar{p}(-1+\bar{\tau}_{2})=1.$ 

Note that we have already defined $\bar{L}_{2}$ and $\bar{L}_{5}$.
Set $\bar{y}_{9}\in C([0,\bar{L}_{2}+\bar{L}_{5}],\mathbb{R})$ so
that 
\[
\bar{y}_{9}(t)=\bar{p}\left(t-1+\bar{\tau}_{2}\right)\quad\mbox{for all }t\in\left[0,\bar{L}_{2}+\bar{L}_{5}\right].
\]
It is clear that $\bar{y}_{9}(0)=\bar{p}(-1+\bar{\tau}_{2})=1.$ Now
recall from \eqref{barp2} that there exists an interval of length
$\bar{L}_{2}+\bar{L}_{5}$ on which $\bar{p}$ decreases from $1$
to $-1$. This fact and Proposition \ref{prop:monotonicity property}.(i)
together imply that if $\bar{p}$ decreases from $1$ to $-1$ on
any subinterval $I$ of $\mathbb{R}$, then the length of $I$ is
$\bar{L}_{2}+\bar{L}_{5}$. Thus $\bar{y}_{9}(\bar{L}_{2}+\bar{L}_{5})=\bar{p}(-1+\bar{\tau}_{2}+\bar{L}_{2}+\bar{L}_{5})=-1.$

Set $\bar{\tau}_{3}=\bar{\tau}_{2}+\bar{L}_{2}+\bar{L}_{5}$. Then,
by our last result, $\bar{p}\left(-1+\bar{\tau}_{3}\right)=-1.$

Fix $\bar{L}_{3}>0$ to be time that $\bar{p}$ needs to decrease
from $-1$ to $-1-\varepsilon$. As we have mentioned in the previous
paragraph, Proposition \ref{prop:monotonicity property}.(i) guarantees
that $\bar{L}_{3}$ is well-defined. Choose $\bar{y}_{10}$ to be
the continuous function on $[0,\bar{L}_{3}]$ defined by 
\[
\bar{y}_{10}(t)=\bar{p}\left(t-1+\bar{\tau}_{3}\right)\quad\mbox{for all }t\in\left[0,\bar{L}_{3}\right].
\]
Then 
\begin{equation}
\bar{y}_{10}(0)=\bar{p}\left(-1+\bar{\tau}_{3}\right)=-1\quad\mbox{and}\quad\bar{y}_{10}\left(\bar{L}_{3}\right)=\bar{p}\left(-1+\bar{\tau}_{3}+\bar{L}_{3}\right)=-1-\varepsilon.\label{barp6.5}
\end{equation}

We have already verified that $\bar{y}_{8}$, $\bar{y}_{9}$ and $\bar{y}_{10}$
fulfill the conditions given in (H4). It remains to show the conditions
listed in (H5): 
\begin{equation}
\bar{y}_{8}\left(t\right)\in(1,1+\varepsilon)\mbox{ for }t\in\left(0,\bar{L}_{4}\right),\quad\bar{y}_{9}\left(t\right)\in\left(-1,1\right)\mbox{ for }t\in\left(0,\bar{L}_{2}+\bar{L}_{5}\right)\label{barp7}
\end{equation}
and 
\begin{equation}
\bar{y}_{10}\left(t\right)\in\left(-1-\varepsilon,-1\right)\mbox{ for }t\in\left(0,\bar{L}_{3}\right).\label{barp8}
\end{equation}
Note that $\bar{\tau}_{3}+\bar{L}_{3}$ is arbitrarily close to $\tau_{3}+L_{3}=\omega$.
Hence, by property \eqref{property5-1}, we may assume that 
\begin{equation}
\bar{p}\left(t\right)<-1-\varepsilon\mbox{ for all }t\mbox{ in }\left(-2+\bar{\tau}_{2}-\bar{L}_{4},-2+\bar{\tau}_{3}+\bar{L}_{3}\right].\label{barp8.5}
\end{equation}
We see from this and from the definitions of $\bar{y}_{8}$, $\bar{y}_{9}$
and $\bar{y}_{10}$ that $\bar{y}_{i}$ is a solution of $\dot{y}=-y-K$
for all $i\in\left\{ 8,9,10\right\} $. Hence the functions $\bar{y}_{8}$,
$\bar{y}_{9}$ and $\bar{y}_{10}$ are strictly decreasing on their
domains. Looking at the boundary values of $\bar{y}_{8}$, $\bar{y}_{9}$
and $\bar{y}_{10}$, it is clear that \eqref{barp7} and \eqref{barp8}
are satisfied.

5. \emph{The function} \emph{$\bar{y}_{7}\in C\left([0,\bar{L}_{3}],\mathbb{R}\right)$.}
By the last step of the proof, if $\bar{p}$ decreases from $-1$
to $-1-\varepsilon$ on an interval $J$, then the length of $J$
is $\bar{L}_{3}$. Now recall from \eqref{barp2} and \eqref{barp5}
that 
\[
\bar{p}\left(-2+\bar{\tau}_{1}+\bar{L}_{2}\right)=-1\quad\mbox{and}\quad\bar{p}\left(-2+\bar{\tau}_{2}-\bar{L}_{4}\right)=-1-\varepsilon.
\]
Hence necessarily 
\begin{equation}
\bar{\tau}_{2}=\bar{\tau}_{1}+\bar{L}_{2}+\bar{L}_{3}+\bar{L}_{4},\label{barp9}
\end{equation}
and the length of $\left[-1+\bar{\tau}_{1}+\bar{L}_{2},-1+\bar{\tau}_{2}-\bar{L}_{4}\right]$
is $\bar{L}_{3}$. Suppose that the function $\bar{y}_{7}\in C\left([0,\bar{L}_{3}],\mathbb{R}\right)$
is defined by 
\[
\bar{y}_{7}\left(t\right)=\bar{p}\left(t-1+\bar{\tau}_{1}+\bar{L}_{2}\right)\quad\mbox{for }t\in\left[0,\bar{L}_{3}\right].
\]
Then $\bar{y}_{7}$ satisfies the boundary conditions in (H4): 
\[
\bar{y}_{7}(0)=\bar{p}\left(-1+\bar{\tau}_{1}+\bar{L}_{2}\right)=\bar{\theta}_{5}\quad\mbox{and}\quad\bar{y}_{7}\left(\bar{L}_{3}\right)=\bar{p}\left(-1+\bar{\tau}_{2}-\bar{L}_{4}\right)=\bar{\theta}_{6},
\]
see \eqref{barp3} and \eqref{barp6}. As $\bar{p}_{0}$ is arbitrarily
close to $p_{0}$, it is clear that $\bar{y}_{7}(t)\in(1,1+\varepsilon)$
for all $t$ in $\left(0,\bar{L}_{3}\right)$ -- as required by (H5).

6. \emph{The functions $\bar{y}_{2}$, $\bar{y}_{3},$ $\bar{y}_{4}$
and the parameters $\bar{\theta}_{2},\bar{\theta}_{3}$.} Recall from
\eqref{barp0} and \eqref{barp2} that 
\[
\bar{p}\left(-2+\bar{L}_{1}\right)=1+\varepsilon\quad\mbox{and}\quad\bar{p}\left(-2+\bar{\tau}_{1}-\bar{L}_{5}\right)=1,
\]
that is, $\bar{p}$ decreases from $1+\varepsilon$ to $1$ on $\left[-2+\bar{L}_{1},-2+\bar{\tau}_{1}-\bar{L}_{5}\right]$.
Recall from the definition of $\bar{\tau}_{1}$ and $\bar{\tau}_{2}$
that $\bar{p}$ decreases from $1+\varepsilon$ to $1$ also on $\left[-1+\bar{\tau}_{1},-1+\bar{\tau}_{2}\right]$.
Necessarily, the length of the interval $\left[-2+\bar{L}_{1},-2+\bar{\tau}_{1}-\bar{L}_{5}\right]$
equals the length of $\left[-1+\bar{\tau}_{1},-1+\bar{\tau}_{2}\right]$,
which is $\bar{L}_{2}+\bar{L}_{3}+\bar{L}_{4}$ by \eqref{barp9}.
In consequence, $\bar{\tau}_{1}=\sum_{i=1}^{5}{\bar{L}_{i}}$ and
the length of $\left[-1+\bar{L}_{1},-1+\bar{\tau}_{1}-\bar{L}_{5}\right]$
is also $\bar{L}_{2}+\bar{L}_{3}+\bar{L}_{4}$. We use this property
to introduce $\bar{y}_{2}\in C([0,\bar{L}_{2}],\mathbb{R})$, $\bar{y}_{3}\in C([0,\bar{L}_{3}],\mathbb{R})$
and $\bar{y}_{4}\in C([0,\bar{L}_{4}],\mathbb{R})$ as restrictions
of $\bar{p}$ to subintervals of $\left[-1+\bar{L}_{1},-1+\bar{\tau}_{1}-\bar{L}_{5}\right]$:
\begin{eqnarray*}
\bar{y}_{2}(t) & = & \bar{p}\left(t-1+\bar{L}_{1}\right)\quad\mbox{for }t\in\left[0,\bar{L}_{2}\right],\\
\bar{y}_{3}(t) & = & \bar{p}\left(t-1+\bar{L}_{1}+\bar{L}_{2}\right)\quad\mbox{for }t\in\left[0,\bar{L}_{3}\right],\\
\bar{y}_{4}(t) & = & \bar{p}\left(t-1+\bar{L}_{1}+\bar{L}_{2}+\bar{L}_{3}\right)\quad\mbox{for }t\in\left[0,\bar{L}_{4}\right].
\end{eqnarray*}
In addition, let
\[
\bar{\theta}_{2}=\bar{y}_{2}\left(\bar{L}_{2}\right)=\bar{y}_{3}(0)\quad\mbox{and}\quad\bar{\theta}_{3}=\bar{y}_{3}\left(\bar{L}_{3}\right)=\bar{y}_{4}(0).
\]
 It is clear that the functions $\bar{y}_{2}$, $\bar{y}_{3},$ $\bar{y}_{4}$
satisfy the boundary conditions required by (H4) because 
\[
\bar{y}_{2}\left(0\right)=\bar{p}\left(-1+\bar{L}_{1}\right)=\bar{y}_{1}\left(\bar{L}_{1}\right)=\bar{\theta}_{1}
\]
(see Step 2 of this proof) and 
\[
\bar{y}_{4}\left(\bar{L}_{4}\right)=\bar{p}\left(-1+\bar{\tau}_{1}-\bar{L}_{5}\right)=\bar{y}_{5}\left(0\right)=\bar{\theta}_{4}
\]
(see \ref{barp3}). Note that $\bar{y}_{i}(t)>1+\varepsilon$ for
all $t$ in the domains of $\bar{y}_{i},i\in\{2,3,4\}$, if and only
if $\bar{p}\left(t\right)>1+\varepsilon$ for $t\in\left[-1+\bar{L}_{1},-1+\bar{\tau}_{1}-\bar{L}_{5}\right]$.
The last inequality holds if $\left\Vert \bar{p}_{0}-p_{0}\right\Vert $
is small enough, as $p\left(t\right)>1+\varepsilon$ for $t\in\left[-1+L_{1},-1+\tau_{1}-L_{5}\right]$.

7. \emph{The proof of the equality $\bar{\omega}=\bar{\tau}_{3}+\bar{L}_{3}$.
}As $\bar{\tau}_{3}$, $\bar{L}_{3}$ and $\bar{\omega}$ are arbitrarily
close to $\tau_{3}$, $L_{3}$ and $\omega=\tau_{3}+L_{3}$, we see
that $\bar{\omega}$ is arbitrarily close to $\bar{\tau}_{3}+\bar{L}_{3}$.
As $\bar{p}\left(-1\right)=1+\varepsilon$, the symmetry property
\eqref{symmetry property} implies that $\bar{p}\left(-1+\bar{\omega}\right)=-1-\varepsilon$.
We have also mentioned that $\bar{p}\left(-1+\bar{\tau}_{3}+\bar{L}_{3}\right)=-1-\varepsilon$,
see \eqref{barp6.5}. If\emph{ }$\bar{\omega}\neq\bar{\tau}_{3}+\bar{L}_{3}$\emph{,}
then there exists $\xi$ between $\bar{\omega}$ and\emph{ }$\bar{\tau}_{3}+\bar{L}_{3}$
such that 
\[
\bar{p}\left(-1+\xi\right)=-1-\varepsilon\quad\mbox{and}\quad\dot{\bar{p}}\left(-1+\xi\right)=0
\]
 (here we use the monotonicity property described in Proposition \ref{prop:monotonicity property}.(i)).
Then necessarily 
\[
f_{\bar{K}}(\bar{p}\left(-2+\xi\right))=\dot{\bar{p}}\left(-1+\xi\right)+\bar{p}\left(-1+\xi\right)=-1-\varepsilon.
\]
This result contradicts the fact that $f_{\bar{K}}(\bar{p}\left(-2+\xi\right))$
is arbitrarily close to 
\[
f_{\bar{K}}\left(\bar{p}\left(-2+\bar{\tau}_{3}+\bar{L}_{3}\right)\right)=-K
\]
(see \eqref{barp8.5}). So $\bar{\omega}=\bar{\tau}_{3}+\bar{L}_{3}$.

Observe that we have verified all the equalities in \eqref{bar tau}.

8. Summing up, $\bar{p}$ is the concatenation of the auxiliary functions
$\bar{y}_{1},\ldots,\bar{y}_{10}$ as given in \eqref{bar p:egyenletek},
the equalities \eqref{bar tau} are satisfied, and all conditions
listed in (H1) and (H3)-(H5) hold. \emph{It remains to verify (H2).}

It is clear from above that 
\[
\bar{L}:=2\bar{L}_{1}+5\bar{L}_{2}+5\bar{L}_{3}+3\bar{L}_{4}+3\bar{L}_{5}
\]
is arbitrarily close to $2L_{1}+5L_{2}+5L_{3}+3L_{4}+3L_{5}=1$ if
$\left|\bar{K}-K^{*}\left(\varepsilon\right)\right|$ and $\left\Vert \bar{p}_{0}-p_{0}\right\Vert $
are small enough. We complete the proof by showing that $\bar{L}=1$.
Using the symmetry property \eqref{symmetry property} and the equations
\eqref{bar p:egyenletek}-\eqref{bar tau}, we calculate that 
\begin{align}
\bar{p}\left(t-1+\bar{L}\right) & =-\bar{y}_{8}(t)\quad\mbox{for }t\in\left[0,\bar{L}_{4}\right],\notag\\
\bar{p}\left(t-1+\bar{L}+\bar{L}_{4}\right) & =-\bar{y}_{9}(t)\quad\mbox{for }t\in\left[0,\bar{L}_{2}+\bar{L}_{5}\right],\notag\\
\bar{p}\left(t-1+\bar{L}+\bar{L}_{2}+\bar{L}_{4}+\bar{L}_{5}\right) & =-\bar{y}_{10}(t)\quad\mbox{for }t\in\left[0,\bar{L}_{3}\right],\notag\\
\bar{p}\left(t-1+\bar{L}+\bar{L}_{2}+\bar{L}_{3}+\bar{L}_{4}+\bar{L}_{5}\right) & =\bar{y}_{1}(t)\quad\mbox{for }t\in\left[0,\bar{L}_{1}\right].\notag
\end{align}
Recall that $-\bar{y}_{8},-\bar{y}_{9},-\bar{y}_{10}$ and $\bar{y}_{1}$
are solutions of $\dot{y}=-y+K$. Hence, by the above equalities,
\[
\dot{\bar{p}}(t)=-\bar{p}\left(t\right)+K\quad\mbox{for }t\in\left[-1+\bar{L},-1+\bar{L}+\bar{\tau}_{1}\right],
\]
which is possible only if
\[
\bar{p}(t)\geq1+\varepsilon\quad\mbox{for }t\in\left[-2+\bar{L},-2+\bar{L}+\bar{\tau}_{1}\right].
\]
 We already know that $\bar{p}(t)<-1-\varepsilon$ for $t\in(-1+\bar{\tau}_{1},-1+\bar{\omega})$.
Also observe that 
\[
\bar{p}(t)<-1-\varepsilon\quad\mbox{for }t\in\left[-1-\bar{L}_{3},-1\right)
\]
 as $\bar{p}(t-1-\bar{L}_{3})=-\bar{y}_{10}(t)$ for $t\in[0,\bar{L}_{3}]$.
As $\bar{L}$ is arbitrarily close to $1$, necessarily $\bar{L}=1$.

9. It follows from Proposition \ref{prop:per sol fixed point} that
$\bar{L}_{2}$ is the fixed point of $L_{2}\mapsto F(L_{2},K,\varepsilon)$. 
\end{proof}

\section{The proof of Theorem \ref{main thm} }

\begin{proof}[Proof of Theorem \ref{main thm}] 
\emph{Step 1.}  According to Proposition \ref{prop:bifurcation}, if $\varepsilon>0$
is sufficiently small, then there are $K^{*}\in\left(6.5,7\right)$
and $L_{2}^{*}\in(0,\widehat{L}_{2}(K^{*},\varepsilon))$ such that
$F_{\varepsilon}$ undergoes a saddle-node bifurcation at $(L_{2}^{*},K^{*})$:
one can give a constant $\delta_{1}>0$ such that 
\begin{itemize}
\item if $K\in(K^{*}-\delta_{1},K^{*})$, then  $F_{\varepsilon}(\cdot,K)$ has no fixed points close to $L_{2}^{*}$,
\item $L_{2}^{*}$ is an isolated fixed point of $F_{\varepsilon}\left(\cdot,K^{*}\right)$, 
\item and  if $K\in(K^{*},K^{*}+\delta_{1})$, then $F_{\varepsilon}(\cdot,K)$ has exactly two fixed points (converging to $L_{2}^{*}$ as $K\rightarrow K^{*}$). 
\end{itemize}

By  Corollary \ref{p is a sol},
the fixed points of $F_{\varepsilon}(\cdot,K)$ yield periodic solutions if $\varepsilon>0$ is
small enough. Thereby we obtain one periodic orbit for $K=K^{*}$, and two different ones for each $K\in(K^{*},K^{*}+\delta_{1})$. Let $p\colon\mathbb{R}\rightarrow\mathbb{R}$ denote the periodic solution given for the bifurcation parameter $K^{*}$.  Corollary \ref{p is a sol} implies that
the initial segments of both periodic solutions corresponding to parameters $K\in(K^{*},K^{*}+\delta_{1})$ converge to $p_0$ as $K\rightarrow K^{*}$.

\emph{Step 2.}  Proposition \ref{prop:no other per sol} gives a constant $\delta_{2}>0$
and an open neighborhood $N$ of $p_{0}$ in the hyperplane
\[
H=\left\{ \varphi\in C\colon\ \varphi\left(-1\right)=1+\varepsilon\right\} 
\]
such that if $K\in(K^{*}-\delta_{2},K^{*}+\delta_{2})$ and $\bar{p}$
is a periodic solution with $\bar{p}_{0}\in N$, then $\bar{p}_{0}$
derives from a fixed point of $F_{\varepsilon}(\cdot,K)$ as in Corollary
\ref{p is a sol}. 

Consider a sufficiently small open neighborhood $B$ of $p_{0}$ in the phase space $C$, and the standard Poincar\'e
map $\mathcal{P}$ from $B$ to H with fixed point $p_{0}$. (The existence of such $\mathcal{P}$ can be shown 
using the implicit function theorem and the fact that $\mathcal{O}=\left\{ p_{t}:t\in\mathbb{R}\right\} $ intersects $H$ transversally, see \cite{Diekmann,Lani-Wayda}
and Appendix I in \cite{Krisztin-Walther-Wu}.) As $\mathcal{P}$
depends continuously on $\varphi\in C$ and on the right hand side
of \eqref{Eq}, we may assume that $\mathcal{P}$ maps $B$ into the neighborhood $N$
for all $K\in(K^{*}-\delta_{2},K^{*}+\delta_{2})$. This means that
if $\bar{p}$ is a periodic solution with segments in $B$, then, by the translation of time, it
derives from a fixed point of $F_{\varepsilon}(\cdot,K)$.

\emph{Step 3.} 
Choose $\delta\in(0,\min\{\delta_{1},\delta_{2}\})$ so small that
for $K\in(K^{*},K^{*}+\delta)$, both periodic solutions given in Step 1 have initial segments in $B$. This is
possible as the initial functions of the periodic solutions converge to $p_0$ as $K\rightarrow K^{*}$.
The main theorem of the paper holds with this
constant $\delta$ and neighborhood $B$. It is clear from our construction
that all periodic solutions in question are of large amplitude, i.e.,
they oscillate about both unstable fixed points of $f_{K}$. \end{proof}

As the bifurcation of the large-amplitude periodic orbits corresponds
to the bifurcation of the fixed points of $F_{\varepsilon}$, we see
from Corollary \ref{cor: limit of K} that $K^{*}$ tends to $K_{0}$
as $\varepsilon\rightarrow0^{+}$.

\appendix
\section{~}
\def\appendixname{}

In the Appendix we examine the partial derivatives of $F\colon U\rightarrow\mathbb{R}$
and work with the assumption that 
\begin{lyxlist}{00.00.0000}
\item [{(H6)}] $L_{i}$, $i\in\{1,3,4,5\}$, and $\theta_{i},$ $1\leq i\leq6$,
are defined by (\ref{L3})-(\ref{theta3}) on $U$.
\end{lyxlist}
We will use notation $O$ as discussed before the proof of Proposition
\ref{prop:Fsorfejtes}. 
\begin{rem}
\label{rem:techincal} Let $\alpha\in\mathbb{R}$. Assume that $\theta_{6}$
is defined by \eqref{theta6} on $U$. Recall from Proposition \ref{prop:Fsorfejtes}.(i)
that $\theta_{6}=1+O(\varepsilon)$ as $\varepsilon\rightarrow0^{+}$.
Therefore, by the binomial expansion, 
\begin{align}
\left(K+\theta_{6}\right)^{\alpha} & =\left(K+1\right)^{\alpha}\left(1+\frac{\theta_{6}-1}{K+1}\right)\notag^{\alpha}\nonumber \\
 & =\left(K+1\right)^{\alpha}\sum_{n=0}^{\infty}{\alpha \choose n}\left(\frac{\theta_{6}-1}{K+1}\right)^{n}\label{*}\\
 & =\left(K+1\right)^{\alpha}+O(\varepsilon)\quad\mbox{as }\varepsilon\rightarrow0^{+}.\notag\nonumber 
\end{align}
Similarly, 
\begin{equation}
(K-1-\varepsilon)^{\alpha}=(K-1)^{\alpha}\left(1-\frac{\varepsilon}{K-1}\right)^{\alpha}=\left(K-1\right)^{\alpha}+O(\varepsilon)\label{**}
\end{equation}
 if $K\in\left(6.5,7\right)$ and $\varepsilon\rightarrow0^{+}$.\end{rem}
\begin{prop}
\label{prop:Kderivalt} Assume (H6). Then $F$ is continuously differentiable
on $U$ with respect to $K$. Furthermore, 
\begin{equation}
\frac{\partial}{\partial K}F(L_{2},K,\varepsilon)=1-e^{-\frac{1}{2}}\left(\frac{3}{2}\sqrt{\frac{K+1}{K-1}}-\frac{1}{2}\sqrt{\left(\frac{K+1}{K-1}\right)^{3}}\right)+\frac{2}{(K-1)^{2}}+O(\varepsilon)\label{der of F wr K}
\end{equation}
as $\varepsilon\rightarrow0^{+}$.\end{prop}
\begin{proof}
\textit{\emph{We explicitly calculate the derivatives of the three
main terms of $F$ with respect to $K$ in order to see that they
are continuous on $U.$ Meanwhile, we show that}}

\textit{(i)} $\partial\theta_{6}/\partial K=O(\varepsilon)$, $\partial L_{4}/\partial K=O(\varepsilon)$
and thus 
\[
\frac{\partial}{\partial K}\left(\frac{K(K+1)}{\varepsilon}\left(1-(1-L_{4})e^{L_{4}}\right)\right)=O(\varepsilon),
\]

\textit{(ii)} 
\[
\frac{\partial\theta_{3}}{\partial K}=1-e^{-\frac{1}{2}}\left(\frac{3}{2}\sqrt{\frac{K+1}{K-1}}-\frac{1}{2}\sqrt{\left(\frac{K+1}{K-1}\right)^{3}}\right)+O(\varepsilon),
\]

\emph{(iii)} and 
\[
\frac{\partial}{\partial K}\left((1+\varepsilon)\frac{K+\theta_{6}}{K-1}e^{-L_{2}}\right)=\frac{-2}{(K-1)^{2}}+O(\varepsilon)\mbox{ as }\varepsilon\rightarrow0^{+}.
\]

For notational simplicity, let $'$ denote differentiation with respect
to $K$ in this proof.

\textit{(i)} The derivative of $\theta_{6}$ with respect to $K$
is 
\begin{equation}
\theta_{6}'=\frac{\varepsilon(1+\varepsilon)}{(K-1)^{2}}e^{-L_{2}}+\frac{2K-1-\varepsilon}{\varepsilon}\ln{\frac{K-1}{K-1-\varepsilon}}-\frac{2K-1}{K-1}.
\end{equation}
As $L_{2}\in(-\varepsilon,\varepsilon)$, it is clear that 
\[
\frac{\varepsilon(1+\varepsilon)}{(K-1)^{2}}e^{-L_{2}}=O(\varepsilon).
\]
Using first (\ref{sorfejtesln}) and then \eqref{**} with $\alpha=-1$,
we deduce that 
\[
\frac{2K-1-\varepsilon}{\varepsilon}\ln{\frac{K-1}{K-1-\varepsilon}}=\frac{2K-1-\varepsilon}{K-1-\varepsilon}+O(\varepsilon)=\frac{2K-1}{K-1}+O(\varepsilon).
\]
It follows from above that $\theta_{6}'=O(\varepsilon)$. 

Next observe that the derivative of $L_{4}$ with respect to $K$
is 
\[
L_{4}'=\frac{1+\theta_{6}'}{K+\theta_{6}}-\frac{1}{K+1}.
\]
Applying \eqref{*} with $\alpha=-1,$ we obtain that $(K+\theta_{6})^{^{-1}}=(K+1)^{^{-1}}+O(\varepsilon).$
Using $\theta_{6}'=O(\varepsilon)$ and the last result, we conclude
that $L_{4}'=O(\varepsilon)$. 

In order to complete the proof of (i), let us examine 
\[
\left(\frac{K(K+1)}{\varepsilon}\left(1-(1-L_{4})e^{L_{4}}\right)\right)',
\]
that is 
\[
\frac{2K+1}{\varepsilon}\left(1-(1-L_{4})e^{L_{4}}\right)+\frac{K(K+1)}{\varepsilon}L_{4}L_{4}'e^{L_{4}}.
\]
By Proposition \ref{prop:Fsorfejtes}.(i), the first term of this
expression is $O(\varepsilon)$. The second term is also $O(\varepsilon)$
because $L_{4}=O(\varepsilon)$ (see again Proposition \ref{prop:Fsorfejtes}.(i))
and $L_{4}'=O(\varepsilon)$. 

\textit{(ii)} Recall that $\theta_{3}$ is defined by \eqref{theta3}.
As $\theta_{3}$ a function of $\theta_{2}$, first we differentiate
$\theta_{2}$ with respect to $K$ using formula \eqref{theta2}:
\begin{align*}
\theta_{2}' & =e^{-L_{2}}+\frac{1}{\varepsilon}\left((1+\varepsilon)L_{2}e^{-L_{2}}+e^{-L_{2}}-1\right)\\
 & \quad-e^{-\frac{1}{2}}\left(\left(\frac{K+\theta_{6}}{K-1-\varepsilon}\right)^{\frac{3}{2}}-\frac{3}{2}\frac{(K-1)(K+\theta_{6})^{\frac{3}{2}}}{(K-1-\varepsilon)^{\frac{5}{2}}}\right)\\
 & \quad+\frac{3}{2}e^{-\frac{1}{2}}\frac{(K-1)(K+\theta_{6})^{\frac{1}{2}}\left(1+\theta_{6}'\right)}{(K-1-\varepsilon)^{\frac{3}{2}}}.
\end{align*}
Using the inequality $|L_{2}|<\varepsilon,$ result (\ref{biz6})
with $t=L_{2}$, Remark \ref{rem:techincal} with various $\alpha$-s
and statement (i) of this proposition, we get that 
\begin{equation}
\theta_{2}'=1-e^{-\frac{1}{2}}\left(\frac{3}{2}\sqrt{\frac{K+1}{K-1}}-\frac{1}{2}\sqrt{\left(\frac{K+1}{K-1}\right)^{3}}\right)+O(\varepsilon).\label{theta2deriv}
\end{equation}

Now we differentiate the second term of $\theta_{3}$ with respect
to $K$. We get from (\ref{L3}) that 
\begin{equation}
L_{3}'=-\frac{\varepsilon}{\left(K-1\right)\left(K-1-\varepsilon\right)}\quad\mbox{and}\quad\left(e^{-L_{3}}\right)'=\frac{\varepsilon}{(K-1)^{2}}.\label{L3deriv}
\end{equation}
Therefore 
\[
\left(\frac{K}{\varepsilon}\left((1+\varepsilon)L_{3}e^{-L_{2}-L_{3}}+e^{-L_{3}}-1\right)\right)'
\]
 is equivalent to
\begin{alignat*}{1}
 & \frac{1}{\varepsilon}\left((1+\varepsilon)L_{3}e^{-L_{2}-L_{3}}+e^{-L_{3}}-1\right)\\
 & +\frac{K}{(K-1)^{2}}\left(1-(1+\varepsilon)e^{-L_{2}}+L_{3}(1+\varepsilon)e^{-L_{2}}\right).
\end{alignat*}
Using that $L_{2}=O(\varepsilon)$, $L_{3}=O(\varepsilon)$ and applying
(\ref{biz9}) with $t=L_{3},$ we deduce that 
\begin{equation}
\left(\frac{K}{\varepsilon}\left((1+\varepsilon)L_{3}e^{-L_{2}-L_{3}}+e^{-L_{3}}-1\right)\right)'=O(\varepsilon).\label{biz11}
\end{equation}

The last term of $\theta_{3}$ is 
\[
\frac{K^{2}}{\varepsilon^{2}}(K-1)\left(1-\left(1+L_{3}+\frac{L_{3}^{2}}{2}\right)e^{-L_{3}}\right).
\]
Its derivative with respect to $K$ is 
\[
\frac{3K^{2}-2K}{\varepsilon^{2}}\left(1-\left(1+L_{3}+\frac{L_{3}^{2}}{2}\right)e^{-L_{3}}\right)-\frac{K^{2}L_{3}^{2}}{2\varepsilon(K-1)}.
\]
Since $L_{3}=O(\varepsilon)$, and (\ref{biz10}) holds for $t=L_{3}$,
we conclude that 
\begin{equation}
\left(\frac{K^{2}}{\varepsilon^{2}}(K-1)\left(1-\left(1+L_{3}+\frac{L_{3}^{2}}{2}\right)e^{-L_{3}}\right)\right)'=O(\varepsilon).\label{biz12}
\end{equation}

Statement (ii) follows immediately from \eqref{theta3}, the fact
that $L_{3}=O\left(\varepsilon\right)$, the boundedness of $\theta_{2}$
(see the proof of Proposition \ref{prop:Fsorfejtes}) and from (\ref{theta2deriv})-(\ref{biz12}). 

\textit{(iii)} It is clear that 
\[
\left((1+\varepsilon)\frac{K+\theta_{6}}{K-1}e^{-L_{2}}\right)'=\frac{1+\varepsilon}{K-1}e^{-L_{2}}\left(1+\theta_{6}'-\frac{K+\theta_{6}}{K-1}\right).
\]
Since $L_{2}=O(\varepsilon)$, $\theta_{6}=1+O\left(\varepsilon\right)$
and $\theta_{6}'=O(\varepsilon),$ we get statement (iii). 

Looking again at the derivatives of the terms of $F$ calculated above,
we see that they are continuous on $U$. Hence $F$ is continuously
differentiable on $U$ with respect to $K$. Formula \eqref{der of F wr K}
follows from statements (i)-(iii).\end{proof}
\begin{cor}
\label{cor:Kszerintiderivlathatarertek} Assume that $\lim_{n\rightarrow\infty}\varepsilon_{n}=0^{+},$
$(L_{2,n},K_{n},\varepsilon_{n})\in U$ for all $n\geq0$ and $F(L_{2,n},K_{n},\varepsilon_{n})=L_{2,n}\mbox{ for all }n\geq0.$
Then 
\begin{alignat}{1}
\lim_{n\rightarrow\infty}\frac{\partial}{\partial K}F(L_{2,n},K_{n},\varepsilon_{n}) & =1-e^{-\frac{1}{2}}\left(\frac{3}{2}\sqrt{\frac{K_{0}+1}{K_{0}-1}}-\frac{1}{2}\sqrt{\left(\frac{K_{0}+1}{K_{0}-1}\right)^{3}}\right)\label{limit of F_K}\\
 & +\frac{2}{(K_{0}-1)^{2}},\notag
\end{alignat}
where $K_{0}$ is the unique solution of (\ref{ke}) in {[}6.5,7{]}.
This limit is positive. \end{cor}
\begin{proof}
It is clear from Corollary \ref{cor: limit of K} and from Proposition
\ref{prop:Kderivalt} that this limit holds. It remains to verify
that it is positive. Expressing $e^{-1/2}$ from (\ref{ke}), we deduce
that the second term on the right hand side of \eqref{limit of F_K}
is 
\[
\frac{{K_{0}}^{2}-2K_{0}-1}{\sqrt{\left(K_{0}-1\right)\left(K_{0}+1\right)^{3}}}\left(\frac{3}{2}\sqrt{\frac{K_{0}+1}{K_{0}-1}}-\frac{1}{2}\sqrt{\left(\frac{K_{0}+1}{K_{0}-1}\right)^{3}}\right),
\]
that is 
\[
\frac{\left({K_{0}}^{2}-2K_{0}-1\right)(K_{0}-2)}{(K_{0}-1)^{2}(K_{0}+1)}.
\]
This expression is smaller than $(K_{0}-2)/(K_{0}+1)<1$, therefore
\eqref{limit of F_K} is greater than zero. \end{proof}
\begin{prop}
\label{prop:L2 szerinti derivalt} Under assumption (H6), $F$ is
continuously differentiable on $U$ with respect to $L_{2}$. In addition,
\[
\frac{\partial}{\partial L_{2}}F(0,K,\varepsilon)=\frac{K^{2}+8K+2}{2\left(K^{2}-1\right)}+\frac{3}{2}e^{-\frac{1}{2}}\sqrt{\frac{K+1}{K-1}}+O(\varepsilon)
\]
and 
\[
\frac{\partial}{\partial L_{2}}F\left(\widehat{L}_{2},K,\varepsilon\right)=\frac{-K^{2}+6K+2}{2(K-1)}+\frac{3}{2}e^{-\frac{1}{2}}\sqrt{\frac{K+1}{K-1}}+O(\varepsilon)\mbox{ as }\varepsilon\rightarrow0^{+},
\]
where $\widehat{L}_{2}$ is given by \eqref{L2kalap}. \end{prop}
\begin{proof}
The proposition will easily follow from the subsequent three claims.

\emph{Claim 1.} The derivative of the first term of $F$ with respect
to $L_{2}$ is 
\begin{equation}
\frac{\partial}{\partial L_{2}}\left(\frac{K(K+1)}{\varepsilon}\left(1-(1-L_{4})e^{L_{4}}\right)\right)=\frac{K}{\varepsilon}L_{4}\frac{\partial\theta_{6}}{\partial L_{2}},\label{prop1}
\end{equation}
where 
\begin{equation}
\frac{\partial\theta_{6}}{\partial L_{2}}=-(1+\varepsilon)\frac{K-1-\varepsilon}{K-1}e^{-L_{2}}.\label{biz14}
\end{equation}
For $L_{2}=0$, \eqref{prop1} equals 
\begin{equation}
-\frac{K(K-4)}{2(K^{2}-1)}+O(\varepsilon)\mbox{ as }\varepsilon\rightarrow0^{+}.\label{*-1}
\end{equation}
For $L_{2}=\widehat{L}_{2}$, it equals 0.

\emph{Claim 2. }The derivative of $\theta_{3}$ with respect to $L_{2}$
is 
\begin{align}
\frac{\partial\theta_{3}}{\partial L_{2}} & =\frac{3}{2}\frac{(1+\varepsilon)\sqrt{K+\theta_{6}}}{\sqrt{K-1-\varepsilon}}e^{-L_{2}-L_{3}-\frac{1}{2}}-\frac{K}{\varepsilon}(1+\varepsilon)(L_{2}+L_{3})e^{-L_{2}-L_{3}}.\label{prop2}
\end{align}
For $L_{2}=0$, this expression is 
\[
\frac{3}{2}e^{-\frac{1}{2}}\sqrt{\frac{K+1}{K-1}}-\frac{K}{K-1}+O(\varepsilon).
\]
For $L_{2}=\widehat{L}_{2}$, the derivative of $\theta_{3}$ with
respect to $L_{2}$ is 
\[
\frac{3}{2}e^{-\frac{1}{2}}\sqrt{\frac{K+1}{K-1}}-\frac{K(K-2)}{2(K-1)}+O(\varepsilon).
\]

\emph{Claim 3.} The derivative of the third term of $F$ with respect
to $L_{2}$ is 
\begin{equation}
\frac{\partial}{\partial L_{2}}\left((1+\varepsilon)\frac{K+\theta_{6}}{K-1}e^{-L_{2}}\right)=-\frac{1+\varepsilon}{K-1}\left(K+\theta_{6}-\theta_{6}'\right)e^{-L_{2}}.\label{prop3}
\end{equation}
For both $L_{2}=0$ and $L_{2}=\widehat{L}_{2}$, this expression
is 
\[
-\frac{K+2}{K-1}+O(\varepsilon).
\]

Let $'$ now denote differentiation with respect to $L_{2}$.

\textit{The proof of Claim 1.} The derivatives of $\theta_{6}$ and
$L_{4}$ with respect to $L_{2}$ are (\ref{biz14}) and 
\begin{equation}
L_{4}'=\frac{\theta_{6}'}{K+\theta_{6}}.\label{biz13}
\end{equation}
It is clear that 
\[
\left(\frac{K(K+1)}{\varepsilon}\left(1-(1-L_{4})e^{L_{4}}\right)\right)'=\frac{K(K+1)}{\varepsilon}L_{4}L_{4}'e^{L_{4}}.
\]
Replacing $e^{L_{4}}$ by $(K+\theta_{6})/(K+1)$ and using (\ref{biz13}),
we get \eqref{prop1}.

Let $L_{2}=0$. From (\ref{biz14}) we obtain that 
\begin{equation}
\theta_{6}'=-1+O(\varepsilon)\mbox{ as }\varepsilon\rightarrow0^{+}.\label{biz15}
\end{equation}
Next we calculate $L_{4}$ for $L_{2}=0$ using formula (\ref{L4}).
So let $L_{2}=0$. Then 
\[
(1+\varepsilon)\frac{K-1-\varepsilon}{K-1}e^{-L_{2}}=(1+\varepsilon)\left(1-\frac{\varepsilon}{K-1}\right)=1+\frac{K-2}{K-1}\varepsilon-\frac{1}{K-1}\varepsilon^{2}.
\]
As a result, by (\ref{theta6}) and (\ref{biz21}), 
\[
\theta_{6}=1+\frac{K-4}{2(K-1)}\varepsilon+O\left(\varepsilon^{2}\right).
\]
We easily get from this and from (\ref{L4}) that 
\begin{equation}
L_{4}=\ln{\left(1+\frac{\theta_{6}-1}{K+1}\right)}=\frac{K-4}{2(K^{2}-1)}\varepsilon+O(\varepsilon^{2}).\label{biz16}
\end{equation}
Substituting (\ref{biz15}) and (\ref{biz16}) into $KL_{4}\theta_{6}'/\varepsilon$,
we get \eqref{*-1}.

By the definition of $\widehat{L}_{2}$, $L_{4}=0$ if $L_{2}=\widehat{L}_{2}$.
Hence \eqref{prop1} equals $0$ in this case, and the proof of Claim
1 is complete.

\textit{The proof of Claim 2.} By (\ref{theta3}), the first term
of $\theta_{3}$ is $\theta_{2}e^{-L_{3}}$. Its derivative with respect
to $L_{2}$ is 
\[
\left(\theta_{2}e^{-L_{3}}\right)'=\frac{3}{2}\frac{(1+\varepsilon)\sqrt{K+\theta_{6}}}{\sqrt{K-1-\varepsilon}}e^{-L_{2}-L_{3}-\frac{1}{2}}-\frac{K}{\varepsilon}\left(1+\varepsilon\right)L_{2}e^{-L_{2}-L_{3}}.
\]
The derivative of the second term of $\theta_{3}$ with respect to
$L_{2}$ is 
\[
\frac{K}{\varepsilon}\left((1+\varepsilon)L_{3}e^{-L_{2}-L_{3}}+e^{-L_{3}}-1\right)'=-\frac{K}{\varepsilon}(1+\varepsilon)L_{3}e^{-L_{2}-L_{3}}.
\]
The third term of $\theta_{3}$ is independent from $L_{2}$. Summing
up, \eqref{prop2} is verified.

According to Remark \ref{rem:techincal}, 
\begin{equation}
\sqrt{K+\theta_{6}}=\sqrt{K+1}+O(\varepsilon)\label{biz34}
\end{equation}
and
\[
\frac{1}{\sqrt{K-1-\varepsilon}}=\frac{1}{\sqrt{K-1}}+O(\varepsilon).
\]
Also recall that $L_{2}=O(\varepsilon)$ and $L_{3}=O(\varepsilon)$.
Hence, for both $L_{2}=0$ and $L_{2}=\hat{L}_{2}$, 
\begin{equation}
\frac{3}{2}\frac{(1+\varepsilon)\sqrt{K+\theta_{6}}}{\sqrt{K-1-\varepsilon}}e^{-L_{2}-L_{3}-\frac{1}{2}}=\frac{3}{2}e^{-\frac{1}{2}}\sqrt{\frac{K+1}{K-1}}+O(\varepsilon).\label{biz23}
\end{equation}

Using (\ref{sorfejtesln}) for $L_{3}=\ln{(K-1)/(K-1-\varepsilon)}$
and then $(K-1-\varepsilon)^{-1}=(K-1)^{-1}+O(\varepsilon)$, we get
the following for $L_{2}=0$: 
\begin{align*}
\frac{K}{\varepsilon}(1+\varepsilon)\left(L_{2}+L_{3}\right)e^{-L_{2}-L_{3}} & =\frac{K}{\varepsilon}(1+\varepsilon)\left(\frac{\varepsilon}{K-1-\varepsilon}+O\left(\varepsilon^{2}\right)\right)(1+O(\varepsilon))\\
 & =\frac{K}{K-1}+O(\varepsilon).
\end{align*}
Subtracting the last result from (\ref{biz23}), the formula for $\theta_{3}'$
at $L_{2}=0$ follows.

Let $L_{2}=\widehat{L}_{2}$. Then by \eqref{biz22} and by (\ref{sorfejtesln}),
\[
\frac{K}{\varepsilon}(1+\varepsilon)\left(\widehat{L}_{2}+L_{3}\right)e^{-\widehat{L}_{2}-L_{3}}
\]
equals 
\begin{alignat*}{1}
 & \frac{K}{\varepsilon}(1+\varepsilon)\left(\left(\frac{K-4}{2(K-1)}+\frac{1}{K-1}\right)\varepsilon+O\left(\varepsilon^{2}\right)\right)(1+O(\varepsilon))\\
 & =\frac{K\left(K-2\right)}{2(K-1)}+O\left(\varepsilon\right).
\end{alignat*}
 Subtracting this from (\ref{biz23}), we complete the proof of Claim
2.

\textit{The proof of Claim 3. }\textit{\emph{The proof of Claim 3
is similar and easy, so we leave it to the reader.}}

The continuous differentiability of $F$ with respect to $L_{2}$
is obvious from (\ref{prop1})-(\ref{biz14}), (\ref{prop2})-(\ref{prop3})
and from the definition of $F$. The formulas for $F'(0,K,\varepsilon)$
and $F'(\widehat{L}_{2},K,\varepsilon)$  immediately follow from
Claims 1--3.\end{proof}
\begin{cor}
\label{cor:L2derivaltjta} Under hypothesis (H6), 
\begin{equation}
\lim_{\varepsilon\rightarrow0^{+}}\frac{\partial}{\partial L_{2}}F(0,\varphi_{\varepsilon}(0),\varepsilon)>1\label{prop4}
\end{equation}
and 
\begin{equation}
\lim_{\varepsilon\rightarrow0^{+}}\frac{\partial}{\partial L_{2}}F\left(\widehat{L}_{2},\varphi_{\varepsilon}\left(\widehat{L}_{2}\right),\varepsilon\right)<1,\label{prop5}
\end{equation}
where $\varphi_{\varepsilon}$ is the map given by Proposition \ref{prop:ift}. \end{cor}
\begin{proof}
Recall from Corollary \ref{cor: limit of K} that if $\varepsilon\rightarrow0^{+}$,
then $\varphi_{\varepsilon}(0)\rightarrow K_{0}\in\left(6.5,7\right)$.
We know from (\ref{ke}) that 
\[
e^{-\frac{1}{2}}=\frac{K_{0}^{2}-2K_{0}-1}{\sqrt{K_{0}-1}\sqrt{\left(K_{0}+1\right)^{3}}}.
\]
Therefore, by the results of the previous proposition, 
\begin{align*}
\lim_{\varepsilon\rightarrow0^{+}}\frac{\partial}{\partial L_{2}}F(0,\varphi_{\varepsilon}(0),\varepsilon) & =\frac{K_{0}^{2}+8K_{0}+2}{2\left(K_{0}^{2}-1\right)}+\frac{3\left(K_{0}^{2}-2K_{0}-1\right)}{2\left(K_{0}^{2}-1\right)}\\
 & =1+\frac{2K_{0}^{2}+2K_{0}+1}{2\left(K_{0}^{2}-1\right)}.
\end{align*}
As $K_{0}\in\left(6.5,7\right)$, the last quotient is clearly positive,
and the limit above is greater than $1$. 

Similarly, 
\begin{align}
\lim_{\varepsilon\rightarrow0^{+}}\frac{\partial}{\partial L_{2}}F\left(\widehat{L}_{2},\varphi_{\varepsilon}\left(\widehat{L}_{2}\right),\varepsilon\right) & =\frac{-K_{0}^{2}+6K_{0}+2}{2\left(K_{0}-1\right)}+\frac{3\left(K_{0}^{2}-2K_{0}-1\right)}{2\left(K_{0}^{2}-1\right)}\notag\label{limit of derivative}\\
 & =1+\frac{-K_{0}^{3}+6K_{0}^{2}+2K_{0}+1}{2K_{0}^{2}-2}.
\end{align}
Since $0.5K^{2}>2K+1$ for $K>2+\sqrt{6}$, we deduce that 
\[
K^{3}>6.5K^{2}>6K^{2}+2K+1\qquad\mbox{for}\quad K>6.5.
\]
As $K_{0}>6.5$, this means that the last quotient in \eqref{limit of derivative}
is negative, and hence the limit for $L_{2}=\widehat{L}_{2}$ is smaller
than $1$.
\end{proof}
In the next proposition we write that $u(L_{2},K,\varepsilon)\sim v(K,\varepsilon)$
as $\varepsilon\rightarrow0^{+}$ for functions $u$ and $v$ defined
on $U$ if 
\[
\lim_{\substack{K\rightarrow\bar{K},\thinspace\varepsilon\rightarrow0^{+},\thinspace L_{2}\in(-\varepsilon,\varepsilon)}
}\frac{u(L_{2},K,\varepsilon)}{v(K,\varepsilon)}=1.
\]

\begin{prop}
\label{prop:L2L2 szerinti derivalt} Under assumption (H6), $\partial^{2}F/\partial L_{2}^{2}$
is continuous on $U$, and 
\begin{equation}
\frac{\partial^{2}}{\partial L_{2}^{2}}F(L_{2},K,\varepsilon)\sim-\frac{K^{2}}{(K+1)\varepsilon}\mbox{ as }\varepsilon\rightarrow0^{+}.\label{2nd der of F wr L2}
\end{equation}
\end{prop}
\begin{proof}
\textit{\emph{We explicitly calculate the second partial derivatives
of the three main terms of $F$ with respect to $L_{2}$. Meanwhile
we }}verify that 

\emph{(i)}
\[
\frac{\partial^{2}}{\partial L_{2}^{2}}\left(\frac{K(K+1)}{\varepsilon}\left(1-(1-L_{4})e^{L_{4}}\right)\right)\sim\frac{K}{(K+1)\varepsilon}
\]

\emph{(ii)} and 
\[
\frac{\partial^{2}}{\partial L_{2}^{2}}\theta_{3}\sim-\frac{K}{\varepsilon}\mbox{ as }\varepsilon\rightarrow0^{+}.
\]

\emph{(iii)} In addition, we show that 
\[
\frac{\partial^{2}}{\partial L_{2}^{2}}\left((1+\varepsilon)\frac{K+1}{K-1}e^{L_{4}-L_{2}}\right)
\]
is bounded for small positive $\varepsilon$. 

Let us again use the symbol $'$ for differentiation with respect
to $L_{2}$. 

\emph{(i)} It is clear from (\ref{biz14}) that $\theta_{6}''=-\theta_{6}'$.
Also recall from (\ref{biz13}) that $L_{4}'=\theta_{6}'/(K+\theta_{6})$.
Therefore, by formula (\ref{prop1}), we get that 
\begin{align*}
\left(\frac{K(K+1)}{\varepsilon}\left(1-(1-L_{4})e^{L_{4}}\right)\right)'' & =\left(\frac{K}{\varepsilon}L_{4}\theta_{6}'\right)'=\frac{K}{\varepsilon}\theta_{6}'\left(\frac{\theta_{6}'}{K+\theta_{6}}-L_{4}\right)
\end{align*}
which is continuous on $U$. Since $\theta_{6}'=-1+O(\varepsilon)$,
$L_{4}=O(\varepsilon)$ and $(K+\theta_{6})^{-1}=(K+1)^{-1}+O(\varepsilon)$
as $\varepsilon\rightarrow0^{+}$, statement (i) follows.

\emph{(ii)} The second derivative of $\theta_{3}$ with respect to
$L_{2}$ can be calculated from \eqref{prop2}. It is also continuous
on $U$: 
\begin{align*}
\theta_{3}'' & =\frac{3(1+\varepsilon)}{2\sqrt{K-1-\varepsilon}}e^{-L_{2}-L_{3}-\frac{1}{2}}\left(\frac{\theta_{6}'}{2\sqrt{K+\theta_{6}}}-\sqrt{K+\theta_{6}}\right)\\
 & \quad-\frac{K}{\varepsilon}(1+\varepsilon)\left(1-L_{2}-L_{3}\right)e^{-L_{2}-L_{3}}.
\end{align*}
Using the same series expansions as before, we can easily see that
statement (ii) is true. 

\emph{(iii)} One can calculate from \eqref{prop3} and from $\theta_{6}''=-\theta_{6}'$
that the second derivative of third term of $F$ with respect to $L_{2}$
is the next continuous function: 
\begin{align*}
\left((1+\varepsilon)\frac{K+1}{K-1}e^{L_{4}-L_{2}}\right)'' & =\frac{1+\varepsilon}{K-1}\left(K+\theta_{6}-3\theta'_{6}\right)e^{-L_{2}}.
\end{align*}
Since $\theta_{6}\rightarrow1$, $\theta_{6}'\rightarrow-1$ and $L_{2}\rightarrow0$
as $\varepsilon\rightarrow0^{+}$, this expression is bounded for
small $\varepsilon>0$.

The continuity of $\partial^{2}F/\partial L_{2}^{2}$ and \eqref{2nd der of F wr L2}
comes from above and from the definition of $F$.
\end{proof}

~

\textbf{Acknowledgments.} 

This research was supported by the Ministry of Human Capacities,
Hungary grant 20391-3/2018/FEKUSTRAT. This research was supported by the EU-funded Hungarian grant EFOP-3.6.2-16-2017-00015. Gabriella Vas was also supported by the
National Research, Development and Innovation Office of Hungary, Grant No. SNN125119 and K129322.

\section*{References}
\bibliographystyle{elsarticle-num} 
\bibliography{mybibfile}

\begin{thebibliography}{10}
\expandafter\ifx\csname url\endcsname\relax
  \def\url#1{\texttt{#1}}\fi
\expandafter\ifx\csname urlprefix\endcsname\relax\def\urlprefix{URL }\fi
\expandafter\ifx\csname href\endcsname\relax
  \def\href#1#2{#2} \def\path#1{#1}\fi

\bibitem{Diekmann}
O.~Diekmann, S.~A. van Gils, S.~M. Verduyn~Lunel, H.-O. Walther,
  \href{http://dx.doi.org/10.1007/978-1-4612-4206-2}{Delay equations}, Vol. 110
  of Applied Mathematical Sciences, Springer-Verlag, New York, 1995,
  functional, complex, and nonlinear analysis.
\newblock \href {https://doi.org/10.1007/978-1-4612-4206-2}
  {\path{doi:10.1007/978-1-4612-4206-2}}.
\newline\urlprefix\url{http://dx.doi.org/10.1007/978-1-4612-4206-2}

\bibitem{Erneux}
T.~Erneux, Applied delay differential equations, Vol.~3 of Surveys and
  Tutorials in the Applied Mathematical Sciences, Springer, New York, 2009.

\bibitem{HVL}
J.~K. Hale, S.~M. Verduyn~Lunel,
  \href{https://doi.org/10.1007/978-1-4612-4342-7}{Introduction to
  functional-differential equations}, Vol.~99 of Applied Mathematical Sciences,
  Springer-Verlag, New York, 1993.
\newblock \href {https://doi.org/10.1007/978-1-4612-4342-7}
  {\path{doi:10.1007/978-1-4612-4342-7}}.
\newline\urlprefix\url{https://doi.org/10.1007/978-1-4612-4342-7}

\bibitem{Walther2}
H.-O. Walther, \href{https://doi.org/10.1365/s13291-014-0086-6}{Topics in delay
  differential equations}, Jahresber. Dtsch. Math.-Ver. 116~(2) (2014) 87--114.
\newblock \href {https://doi.org/10.1365/s13291-014-0086-6}
  {\path{doi:10.1365/s13291-014-0086-6}}.
\newline\urlprefix\url{https://doi.org/10.1365/s13291-014-0086-6}

\bibitem{Krisztin-1}
T.~Krisztin, \href{http://dx.doi.org/10.1007/s10998-008-5083-x}{Global dynamics
  of delay differential equations}, Period. Math. Hungar. 56~(1) (2008) 83--95.
\newblock \href {https://doi.org/10.1007/s10998-008-5083-x}
  {\path{doi:10.1007/s10998-008-5083-x}}.
\newline\urlprefix\url{http://dx.doi.org/10.1007/s10998-008-5083-x}

\bibitem{Krisztin-Walther}
T.~Krisztin, H.-O. Walther,
  \href{http://dx.doi.org/10.1023/A:1009091930589}{Unique periodic orbits for
  delayed positive feedback and the global attractor}, J. Dynam. Differential
  Equations 13~(1) (2001) 1--57.
\newblock \href {https://doi.org/10.1023/A:1009091930589}
  {\path{doi:10.1023/A:1009091930589}}.
\newline\urlprefix\url{http://dx.doi.org/10.1023/A:1009091930589}

\bibitem{Krisztin-Walther-Wu}
T.~Krisztin, H.-O. Walther, J.~Wu, Shape, smoothness and invariant
  stratification of an attracting set for delayed monotone positive feedback,
  Vol.~11 of Fields Institute Monographs, American Mathematical Society,
  Providence, RI, 1999.

\bibitem{Krisztin-Wu}
T.~Krisztin, J.~Wu, The global structure of an attracting set, in preparation.

\bibitem{Walther}
H.-O. Walther, \href{https://doi.org/10.1007/s10231-005-0170-8}{Bifurcation of
  periodic solutions with large periods for a delay differential equation},
  Ann. Mat. Pura Appl. (4) 185~(4) (2006) 577--611.
\newblock \href {https://doi.org/10.1007/s10231-005-0170-8}
  {\path{doi:10.1007/s10231-005-0170-8}}.
\newline\urlprefix\url{https://doi.org/10.1007/s10231-005-0170-8}

\bibitem{Vas1}
G.~Vas, \href{https://doi.org/10.1016/j.na.2009.01.078}{Asymptotic constancy
  and periodicity for a single neuron model with delay}, Nonlinear Anal.
  71~(5-6) (2009) 2268--2277.
\newblock \href {https://doi.org/10.1016/j.na.2009.01.078}
  {\path{doi:10.1016/j.na.2009.01.078}}.
\newline\urlprefix\url{https://doi.org/10.1016/j.na.2009.01.078}

\bibitem{Krisztin-Vas}
T.~Krisztin, G.~Vas,
  \href{http://dx.doi.org/10.1007/s10884-011-9225-2}{Large-amplitude periodic
  solutions for differential equations with delayed monotone positive
  feedback}, J. Dynam. Differential Equations 23~(4) (2011) 727--790.
\newblock \href {https://doi.org/10.1007/s10884-011-9225-2}
  {\path{doi:10.1007/s10884-011-9225-2}}.
\newline\urlprefix\url{http://dx.doi.org/10.1007/s10884-011-9225-2}

\bibitem{Vas2}
G.~Vas, \href{https://doi.org/10.1016/j.jde.2016.10.031}{Configurations of
  periodic orbits for equations with delayed positive feedback}, J.
  Differential Equations 262~(3) (2017) 1850--1896.
\newblock \href {https://doi.org/10.1016/j.jde.2016.10.031}
  {\path{doi:10.1016/j.jde.2016.10.031}}.
\newline\urlprefix\url{https://doi.org/10.1016/j.jde.2016.10.031}

\bibitem{Krisztin-Vas_2}
T.~Krisztin, G.~Vas, \href{http://dx.doi.org/10.1007/s10884-014-9375-0}{The
  {U}nstable {S}et of a {P}eriodic {O}rbit for {D}elayed {P}ositive
  {F}eedback}, J. Dynam. Differential Equations 28~(3-4) (2016) 805--855.
\newblock \href {https://doi.org/10.1007/s10884-014-9375-0}
  {\path{doi:10.1007/s10884-014-9375-0}}.
\newline\urlprefix\url{http://dx.doi.org/10.1007/s10884-014-9375-0}

\bibitem{Ivanov-Losson}
A.~F. Ivanov, J.~Losson, Stable rapidly oscillating solutions in delay
  differential equations with negative feedback, Differential Integral
  Equations 12~(6) (1999) 811--832.

\bibitem{Lani-Wayda}
B.~Lani-Wayda, \href{https://doi.org/10.1007/BF02218814}{Persistence of
  {P}oincar\'{e} mappings in functional-differential equations (with
  application to structural stability of complicated behavior)}, J. Dynam.
  Differential Equations 7~(1) (1995) 1--71.
\newblock \href {https://doi.org/10.1007/BF02218814}
  {\path{doi:10.1007/BF02218814}}.
\newline\urlprefix\url{https://doi.org/10.1007/BF02218814}

\bibitem{Stoffer}
D.~Stoffer, \href{https://doi.org/10.1007/s10884-006-9068-4}{Delay equations
  with rapidly oscillating stable periodic solutions}, J. Dynam. Differential
  Equations 20~(1) (2008) 201--238.
\newblock \href {https://doi.org/10.1007/s10884-006-9068-4}
  {\path{doi:10.1007/s10884-006-9068-4}}.
\newline\urlprefix\url{https://doi.org/10.1007/s10884-006-9068-4}

\bibitem{Hale}
J.~K. Hale, Ordinary differential equations, 2nd Edition, Robert E. Krieger
  Publishing Co., Inc., Huntington, N.Y., 1980.

\bibitem{Wiggins}
S.~Wiggins, Introduction to applied nonlinear dynamical systems and chaos, 2nd
  Edition, Vol.~2 of Texts in Applied Mathematics, Springer-Verlag, New York,
  2003.

\end{thebibliography}





\end{document}